\documentclass{amsart}

\usepackage{graphicx} % Required for inserting images
\usepackage{accents}

\usepackage[USenglish]{babel}
\usepackage{bbm}
\usepackage{amsmath}
\usepackage{amsfonts}
\usepackage{amssymb}
\usepackage{amsthm}
\usepackage{bussproofs}
\usepackage{enumerate}
\usepackage{enumitem}
\usepackage{mathtools}
\usepackage[hidelinks]{hyperref}
\usepackage{scrextend}
\usepackage{tikz}
\usepackage[autostyle=false, style=english]{csquotes}
\usepackage{commath}

\makeatletter
\newcommand*{\rn}[1]{%
  \expandafter\@rn\csname c@#1\endcsname%
}

\newcommand*{\@rn}[1]{%
  $\ifcase#1\or(i)\or(ii)\or(iii)\or(iv)\or(v)\or(vi)\or(vii)\or(viii)\or(ix)\or(x)%
    \else\@ctrerr\fi$%
}
\AddEnumerateCounter{\rn}{\@rn}{53.13}
\makeatother

  \newbox\gnBoxA
\newdimen\gnCornerHgt
\setbox\gnBoxA=\hbox{$\ulcorner$}
\global\gnCornerHgt=\ht\gnBoxA
\newdimen\gnArgHgt
\def\Godelnum #1{%
\setbox\gnBoxA=\hbox{$#1$}%
\gnArgHgt=\ht\gnBoxA%
\ifnum     \gnArgHgt<\gnCornerHgt \gnArgHgt=0pt%
\else \advance \gnArgHgt by -\gnCornerHgt%
\fi \raise\gnArgHgt\hbox{$\ulcorner$} \box\gnBoxA %
\raise\gnArgHgt\hbox{$\urcorner$}}

\makeatletter
\newcommand{\pushright}[1]{\ifmeasuring@#1\else\omit\hfill$\displaystyle#1$\fi\ignorespaces}
\newcommand{\pushleft}[1]{\ifmeasuring@#1\else\omit$\displaystyle#1$\hfill\fi\ignorespaces}
\makeatother

\newcommand{\PP}{\mathbb{P}}

\newcommand{\Q}{\dot{\mathbb{Q}}}

\newcommand{\forces}{\Vdash}
\newcommand{\res}{\upharpoonright}

\newcommand{\dotieconcat}[2]{\text{\raisebox{.8ex}{$\smallfrown$}}}

\newcommand{\QQ}{\mathbb{Q}}

\newcommand{\dom}{\mathrm{dom}}
\newcommand{\ran}{\mathrm{ran}}
\newcommand{\ZFC}{\mathrm{ZFC}}

\newtheoremstyle{nopoint}% name of the style to be used
  {}{}{\itshape}{}{\bfseries}{}{5pt}{}% Manually specify head

\theoremstyle{plain}
\newtheorem{thm}{Theorem}[section]
\newtheorem{prop}[thm]{Proposition}
\newtheorem{lemm}[thm]{Lemma}

\newtheorem{cor}[thm]{Corollary}
\newtheorem{claim}[thm]{Claim}

\theoremstyle{definition}
\newtheorem{defn}[thm]{Definition}

\newtheorem{rem}[thm]{Remark}
\newtheorem{que}[thm]{Question}
\newtheorem*{que*}{Question}

\theoremstyle{nopoint}
\newtheorem*{lemm*}{Lemma}
\newtheorem*{thm*}{Theorem}
\newtheorem*{rem*}{Remark}

\newtheorem{assump}[thm]{Assumption}
\newcommand{\Col}{\mathrm{Col}}

\newcommand{\Add}{\mathrm{Add}}
\newcommand{\Addinf}{\Add(\infty)}
\newcommand{\Ord}{\mathrm{Ord}}
\newcommand{\CCA}{\mathsf{CCA}}

\newcommand{\LCCA}{\mathsf{LCCA}}

\newcommand{\MM}{\mathsf{MM}}

\newcommand{\crit}{\mathrm{crit}}
\newcommand{\Con}{\mathrm{Con}}
\newcommand{\Hull}{\mathrm{Hull}}

\newcommand{\MP}[3]{#1\mathchar`-\mathsf{MP}_{#2}(#3)}
\newcommand{\allMP}[2]{\mathsf{MP}_{#1}(#2)}
\newcommand{\MPall}[2]{#1\mathchar`-\mathsf{MP}(#2)}
\newcommand{\MPempty}[2]{#1\mathchar`-\mathsf{MP}_{#2}}
\newcommand{\allMPall}[1]{\mathsf{MP}(#1)}
\newcommand{\allMPempty}[1]{\mathsf{MP}_{#1}}
\newcommand{\MPallempty}[1]{#1\mathchar`-\mathsf{MP}}

\newcommand{\SnMP}[2]{\MP{\Sigma_n}{#1}{#2}}
\newcommand{\PnMP}[2]{\MP{\Pi_n}{#1}{#2}}
\newcommand{\SnMPall}[1]{\MPall{\Sigma_n}{#1}}
\newcommand{\PnMPall}[1]{\MPall{\Pi_n}{#1}}
\newcommand{\SnMPallR}{\SnMPall{\mathbb{R}}}
\newcommand{\PnMPallR}{\PnMPall{\mathbb{R}}}
\newcommand{\PnMPallempty}{\MPallempty{\Pi_n}}

\newcommand{\MPprime}[3]{#1\mathchar`-\mathsf{MP}'_{#2}(#3)}

\newcommand{\MPast}[3]{#1\mathchar`-\mathsf{MP}^{\ast}_{#2}(#3)}
\newcommand{\allMPast}[2]{\mathsf{MP}^{\ast}_{#1}(#2)}
\newcommand{\MPastall}[2]{#1\mathchar`-\mathsf{MP}^{\ast}(#2)}
\newcommand{\MPastempty}[2]{#1\mathchar`-\mathsf{MP}^{\ast}_{#2}}

\newcommand{\MPastallempty}[1]{#1\mathchar`-\mathsf{MP}^{\ast}}

\newcommand{\MPminus}[3]{#1\mathchar`-\mathsf{MP}^{-}_{#2}(#3)}

\newcommand{\MPplus}[3]{#1\mathchar`-\mathsf{MP}^{+}_{#2}(#3)}
\newcommand{\allMPplus}[2]{\mathsf{MP}^{+}_{#1}(#2)}

\newcommand{\MPplusallempty}[1]{#1\mathchar`-\mathsf{MP}^{+}}

\newcommand{\bRA}[2]{\undertilde{\mathsf{RA}}_{#1}(#2)}

\renewcommand{\subset}{\subseteq}

\usepackage{comment}
\usepackage{tikz, here}
\usetikzlibrary{cd}
\usepackage{lineno} % line numbers
\title{Separating Maximality Principles}
\author{Takehiko Gappo}
\address{Takehiko Gappo, Institut f\"{u}r Diskrete Mathematik und Geometrie, TU Wien, Wiedner Hauptstra{\ss}e 8-10/104, 1040 Wien, Austria.}
\email{takehiko.gappo@tuwien.ac.at}

\author{Andreas Lietz}
\address{Andreas Lietz, Institut f\"ur Diskrete Mathematik und Geometrie, TU Wien, Wiedner Hauptstra{\ss}e 8-10/104, 1040 Wien, Austria.}
\email{andreas.lietz@tuwien.ac.at}

\subjclass[2020]{03E57, 03E35, 03E55, 03E45}

\keywords{Maximality principles, forcing axioms, large cardinals.}
\date{\today}

\begin{document}
    
    \maketitle

    \begin{abstract}
        We investigate fragments of generic absoluteness principles known as Maximality Principles. We determine the consistency strength of $\SnMP{}{\mathbb R}$ and $\PnMP{}{\mathbb R}$, the boldface Maximality Principle restricted respectively to $\Sigma_n$- and $\Pi_n$-formulas. Further, we show that no implication between $\SnMP{}{\mathbb R}$ and $\PnMP{}{\mathbb R}$ is provable in $\ZFC$. We also establish the consistency, relative to a Woodin cardinal, of the Maximality Principle for $\omega_1$-preserving posets with countable ordinal parameters and prove its consistency strength is bounded below by a Ramsey cardinal. Finally, we resolve questions of Ikegami--Trang and Goodman by separating the Maximality Principle for stationary set preserving posets restricted to $\Sigma_2$-formulas from $\MM^{++}$ in the presence of large cardinals.
    \end{abstract}

    \setcounter{tocdepth}{1}
    \tableofcontents
    
    \section{Introduction}

    %%% NEW INTRO START %%%

    Since Paul Cohen's resolution of Hilbert's first problem, the Continuum Hypothesis, it has been understood that G\"odel's independence phenomenon affects not only specific arithmetic sentences but also natural mathematical questions. Set theorists therefore search for natural extensions of the standard base theory $\mathsf{ZFC}$ which decide as many such questions as possible. One promising direction comes from the very method typically used to establish independence results: forcing. Formalizing the idea that ``a maximal amount of certain forcing has been done" leads to forcing axioms such as Martin's Axiom, the Proper Forcing Axiom and Martin's Maximum, each of which have found far-reaching applications within set theory and beyond.

    A different way to maximizing truth relative to forcing is provided by \emph{Maximality Principles} ($\mathsf{MP}$). These assert that if a statement can be made true by forcing but not subsequently destroyed, then it should hold already. Maximality Principles and forcing axioms are orthogonal in the sense that there is typically no implication between related instances of each of them, but nonetheless they work in tandem. In fact, the Maximality Principle for a given class of forcings usually hold in the standard model of a related forcing axiom, provided the large cardinal used in the construction exhibits some additional correctness properties. For example, in Baumgartner's standard model of the Proper Forcing Axiom, the Maximality Principle for proper forcing holds true, provided that the supercompact cardinal $\kappa$ in the ground model $W$ satisfies $V_\kappa^W\prec W$.

    Weaker fragments of forcing axioms such as the bounded forcing axioms $\mathsf{BPFA}$ and $\mathsf{BMM}$ are by now well understood and have an extensive literature. In this paper, we investigate natural fragments of Maximality Principles obtained by restricting the class of allowed formulas. Below are the main highlights:
    \begin{enumerate}
        \item We calibrate the exact consistency strength of $\Sigma_n\mathchar`-\mathsf{MP}$ and $\Pi_n\mathchar`-\mathsf{MP}$ for all posets with real parameters, showing that they differ. As a consequence, we stablish the consistency of $\Pi_n\mathchar`-\mathsf{MP}\,+\,\neg\Sigma_n\mathchar`-\mathsf{MP}$ for all posets with real parameters, relative to some large cardinals.
        \item We also show that $\Sigma_n\mathchar`-\mathsf{MP}\,+\,\neg\Pi_n\mathchar`-\mathsf{MP}$ for all posets with real parameters is consistent relative to some large cardinals.
        \item We show that $\mathsf{MP}$ for $\omega_1$-preserving posets with countable ordinal parameters is consistent relative to large cardinals. We show that the upper bound of the consistency strength of this statement is a Woodin cardinal and a measurable cardinal above it, while the lower bound is a Ramsey cardinal.
        \item Large cardinal axioms does not imply $\Sigma_2\mathchar`-\mathsf{MP}$ for all posets. Also, $\mathsf{MM}^{++}$ and the existence of a proper class of Woodin cardinals do not imply $\Sigma_2\mathchar`-\mathsf{MP}$ for stationary set preserving posets. This answers questions raised by Ikegami--Trang \cite{IkegamiTrangOnAClassOfMaximilityPrinciples} and Goodman \cite{GoodmanSigmaNCorrectForcingAxioms}.
    \end{enumerate}

    \subsection*{Acknowledgments}

    The first author thanks Fuchino Saka\'e and Francesco Parente for helpful discussion on Maximality Principles. Both authors thank Juliette Kennedy and Jouko V\"a\"an\"anen for letting us know the early history of the Maximality Principle in the Arctic Set Theory Workshop in 2025.
    
    This research was funded in part by the Austrian Science Fund (FWF) [10.55776 /I6087]. The second author was supported by the German Research Foundation (DFG) project 542745507.  For the purpose of open access, the authors have applied a CC BY public copyright license to any Author Accepted Manuscript version arising from this submission.

    \section{Maximality Principles}

    The Maximality Principle was first introduced by Stavi and V\"a\"an\"anen in \cite{StaviVaananenReflectionPrinciplesForTheContinuum}. Although their paper was written in the late 1970s, it was published more than two decades later.\footnote{We refer the readers interested in the early history of the Maximality Principle to \cite{VaananenPursuingLogicWithoutBorders}.} One of the results in \cite{StaviVaananenReflectionPrinciplesForTheContinuum} was a characterization of Martin's Axiom in terms of generic absoluteness. Joan Bagaria independently established the same characterization in \cite{BagariaACharacterizationOfMartinsAxiomInTermsOfAbsoluteness} and had arrived at the idea of the Maximality Principle. Hamkins also rediscovered and further studied the Maximality Principle in \cite{HamkinsASimpleMaximalityPrinciple}, building on an idea of Christophe Chalons, as reformulated by Paul B. Larson. So far, the study of the Maximality Principle has expanded in multiple directions through the contributions of many people, including Asperó \cite{AsperoAMaximalBoundedForcingAxiom}, Leibman \cite{LeibmanPhDThesis}, Hamkins--Woodin \cite{HamkinsWoodinTheNecessaryMaximalityPrinciple}, Fuchs \cite{FuchsClosedMaximalityPrinciples, FuchsCombinedMaximalityPrinciplesUpToLargeCardinals, FuchsSubcompleteForcingPrinciplesAndDefinableWellorders}, Ikegami--Trang \cite{IkegamiTrangOnAClassOfMaximilityPrinciples}, and Minden \cite{MindenCombiningResurrectionAndMaximality}, among others.

    In the papers mentioned above, the terminologies and notations for the variants of the Maximality Principles are not always consistent. Therefore, let us begin with basic definitions, along with several basic observations.
    
    \begin{defn}
        Let $\mathcal{P}$ be a class of forcing posets that is definable in the language of set theory without parameters.
        We say that $\mathcal{P}$ is \emph{iterable} if $\mathcal{P}$ contains the trivial poset and is closed under forcing equivalence, restriction, and two-step iteration.
        We informally say that $\mathcal{P}$ is \emph{transfinitely iterable} if $\mathcal{P}$ is iterable and is closed under iteration of  using some appropriate support.
    \end{defn}

    \begin{defn}
        Let $\mathcal{P}$ be a definable class of posets. Let $\phi(\vec{x})$ be a formula and let $\vec{a}$ be a parameter.
        \begin{enumerate}
            \item $\phi[\vec{a}]$ is \emph{$\mathcal{P}$-forceable} if there are a poset $\mathbb{P}\in\mathcal{P}$ and $p\in\mathbb{P}$ such that $p\forces_{\mathbb{P}}\phi[\vec{a}]$.
            \item $\phi[\vec{a}]$ is \emph{$\mathcal{P}$-persistent} (or \emph{$\mathcal{P}$-necessary}) if for all posets $\mathbb{P}\in\mathcal{P}$ and all $p\in\mathbb{P}$, $p\forces_{\mathbb{P}}\phi[\vec{a}]$.
            \item $\phi[\vec{a}]$ is \emph{$\mathcal{P}$-forceably $\mathcal{P}$-persistent} if the sentence ``$\phi[\vec{a}]$ is $\mathcal{P}$-persistent'' is $\mathcal{P}$-forceable.
        \end{enumerate}
    \end{defn}

    \begin{defn}
        Let $\Gamma$ be a set of formulas, $\mathcal{P}$ be a (definable) class of posets, and $A$ be a set. Then $\MP{\Gamma}{\mathcal{P}}{A}$, the \emph{$\Gamma$-Maximality Principle for $\mathcal{P}$ with parameters in $A$}, is the scheme of formulas asserting that for any formula $\phi(\vec{x})$ in $\Gamma$ and any $\vec{a}\in A^{<\omega}$, if $\phi[\vec{a}]$ is $\mathcal{P}$-forceably $\mathcal{P}$-persistent, then $\phi[\vec{a}]$ holds.

        We omit $\Gamma, \mathcal{P}$ or $A$ from theses notations in the case that
        \begin{itemize}
            \item $\Gamma$ is the set of all formulas,
            \item $\mathcal{P}$ is the class of all posets,
            \item $A = \emptyset$.
        \end{itemize}
        Thus, $\mathsf{MP}$ denotes the original Maximality Principle.
        
        We also use standard abbreviations to refer to classes of posets as follows.
        \begin{enumerate}
            \item ``c.c.c.'' stands for the class of all c.c.c.\ posets.
            \item ``proper'' stands for the class of all proper posets.
            \item ``semiproper'' stands for the class of all semiproper posets.
            \item ``$\sigma$-closed'' stands for the closure of $\sigma$-closed posets under forcing equivalence.
            \item ``subcomp'' stands for the class of all subcomplete posets.
            \item ``SPP'' stands for the class of all stationary preserving posets.
            \item ``$\omega_1$-pres'' stands for the class of all $\omega_1$-preserving posets.
        \end{enumerate}
        For example, the Maximality Principle for c.c.c.\ posets without parameters is denoted by $\allMPempty{\mathrm{c.c.c.}}$.
    \end{defn}

    A parameter set $A$ in the definition of $\MP{\Gamma}{\mathcal{P}}{A}$ cannot be taken arbitrary large. The easy proofs of the next observation can be found in \cite{HamkinsASimpleMaximalityPrinciple} and \cite[Lemma 2.2 and Lemma 2.3]{MindenCombiningResurrectionAndMaximality}.
    
    \begin{prop}\label{prop:MaximalParameterSets}
        Let $\mathcal{P}$ be a class of posets.
        \begin{enumerate}
            \item $\allMPall{A}$ implies $A \subset H_{\omega_1}$.
            \item If $\mathcal{P}$ includes posets that collapse arbitrarily large cardinals to $\omega_1$, then $\allMP{\mathcal{P}}{A}$ implies $A \subset H_{\omega_2}$.
            \item If $\mathcal{P}$ includes posets that increase the continuum arbitrarily large without collapsing cardinals, then $\allMP{\mathcal{P}}{A}$ implies $A \subset H_{2^{\omega}}$.
        \end{enumerate}
    \end{prop}

    For many natural classes of posets $\mathcal{P}$, the largest parameter set $A$ for which $\allMP{\mathcal{P}}{A}$ is consistent (relative to some large cardinal axiom) is known. One of the few exceptions is the class of $\omega_1$-preserving posets. See \S \ref{subsection:omega_1-pres} and Question \ref{que:omega_1-pres}.

    Note that $\MP{\Pi_1}{\mathcal{P}}{A}$ holds by downward absoluteness of $\Pi_1$-sentences. The Maximality Principles for $\Sigma_1$ formulas do not follow from $\mathsf{ZFC}$, because it is equivalent to the Bounded Forcing Axioms.
    
    \begin{thm}[Bagaria \cite{BagariaBoundedForcingAxiomsAsPrinciplesOfGenericAbsoluteness}\footnotemark]\label{BagariaBFAvsMP}
        Let $\mathbb{P}$ be a class of posets that is closed under forcing equivalence and restriction. For an uncountable cardinal $\kappa$, the following are equivalent:
        \begin{enumerate}
            \item $\MP{\Sigma_1}{\mathcal{P}}{H_{\kappa}}$.
            \item $\mathsf{BFA}_{<\kappa}(\mathcal{P})$ holds, i.e.\ for any $\mathbb{P}\in\mathcal{P}$ and any family $\mathcal{A}$ of maximal antichains in the Boolean completion $\mathbb{B}$ of $\mathbb{P}$ such that $\lvert\mathcal{A}\rvert < \kappa$ and $\lvert A\rvert < \kappa$ for all $A\in\mathcal{A}$, there is a filter $G\subset\mathbb{B}$ such that for all $A\in\mathcal{A}$, $A\cap G\neq\emptyset$.
        \end{enumerate}
    \end{thm}
    \footnotetext{Theorem \ref{BagariaBFAvsMP} for the case where $\mathcal{P}$ is the class of c.c.c.\ posets and $\kappa = 2^{\aleph_0}$ is due to Stavi \cite{StaviVaananenReflectionPrinciplesForTheContinuum} and Bagaria \cite{BagariaACharacterizationOfMartinsAxiomInTermsOfAbsoluteness} independently.}
    
    In \cite{BagariaBoundedForcingAxiomsAsPrinciplesOfGenericAbsoluteness}, Theorem \ref{BagariaBFAvsMP} was proved under the assumption that $\kappa$ is a successor of a cardinal of uncountable cofinality, but this assumption can be removed by slightly modifying Bagaria's proof. See \cite{FuchinoGappoParenteGenericAbsolutenessRevisited} for the detail.

    For two classes of posets $\mathcal{P} \subsetneq \mathcal{Q}$, $\allMP{\mathcal{Q}}{A}$ does not imply $\allMP{\mathcal{P}}{A}$ in general. Indeed, they are incompatible in many cases. The following observation does not seem present in the previous literature. Recall that $\undertilde{\delta}^1_2$ denotes the supremum of the order type of a $\undertilde{\Delta}^1_2$ prewellordering of $\mathbb{R}$.

    \begin{thm}\label{thm:PropervsSemiproper}
        Assume that there are unboundedly many Woodin cardinals.
        Then $\MP{\Sigma_2}{\mathrm{semiproper}}{H_{\omega_2}}$ implies $\undertilde{\delta}^1_2 = \omega_2$.
        On the other hand, $\MPempty{\Sigma_2}{\mathrm{proper}}$ implies $\undertilde{\delta}^1_2 < \omega_2$.
    \end{thm}
    \begin{proof}
        Note that the value of $\undertilde{\delta}^1_2$ can only increase and for any $\alpha<\omega_2$, there is a semiproper poset $\PP$ with $V^\PP\models \undertilde{\delta}^1_2\geq\alpha$. The latter can be achieved in several ways, e.g.\ forcing $\mathrm{NS}_{\omega_1}$ to be saturated using a Woodin cardinal so that by a Theorem of Woodin \cite{WoodinPmax}, $\undertilde{\delta}^1_2=\omega_2$ holds in the extension. So we have made $\undertilde{\delta}^1_2\geq\alpha$ persistent in $V^\PP$, hence it is true already in $V$, assuming $\MP{\Sigma_2}{\text{semiproper}}{H_{\omega_2}}$.

        On the other hand, under our large cardinal assumption, the value of $\undertilde{\delta}^1_2$ cannot be changed by proper forcings by a result of Foreman--Magidor \cite{ForemanMagidorCH}\footnote{See also \cite{NeemanZapletal1998} for a stronger version of this.}. Thus in $V^{\Col(\omega_1,\omega_2)}$, the statement $\undertilde{\delta}^1_2<\omega_2$ is true and persistent with respect to proper forcings.
        Hnece it is true in $V$, assuming $\MPempty{\Sigma_2}{\mathrm{proper}}$.
    \end{proof}

    \subsection{First-order expressibility of $\MP{\Gamma}{\mathcal{P}}{A}$}\label{Subsection:FirstOrderExpressibility}

    Officially, $\MP{\Gamma}{\mathcal{P}}{A}$ is always defined as a scheme ranging over all formulas in $\Gamma$ of the meta-theory. In the case that $\Gamma$ is either the set of all $\Sigma_n$ formulas or the set of all $\Pi_n$ formulas, there is a natural way to formalize $\MP{\Gamma}{\mathcal{P}}{A}$ as a first-order statement.
    We point out some subtlety hidden here.
    
    \begin{defn}
        We say that a set of formulas $\Gamma$ \emph{admits $\Gamma$-satisfaction relations} if there is a binary formula $\Phi(u_0, u_1)$ in $\Gamma$ such that for all $x$ and $y$,
        \begin{align*}
        \Phi[x, y] \iff & x = \lceil\phi(\vec{v})\rceil\text{ for some formula $\phi(\vec{v})$ in $\Gamma$},\\
                        & y \text{ is a $\lvert\vec{v}\rvert$-tuple, and }\phi[y]\text{ holds},
        \end{align*}
        where $\lceil\phi(\vec{v})\rceil$ denotes the G\"{o}del number of $\phi(\vec{v})$.
    \end{defn}

    If $\Gamma$ is either the class of all $\Sigma_n$-formulas or one of all $\Pi_n$ formulas, where $n \geq 1$, then $\Gamma$ admits $\Gamma$-satisfaction relations. See Chapter 0 of \cite{KanamoriHigherInfinite}.
    
    \begin{defn}
        Assume that $\Gamma$ admits $\Gamma$-satisfaction relations witnessed by $\Phi$.
        Then $\MPprime{\Gamma}{\mathcal{P}}{A}$ is the following sentence: for all $x$ and $\vec{a} \in A^{<\omega}$, if $\Phi[x, \vec{a}]$ is $\mathcal{P}$-forceably $\mathcal{P}$-persistent, then $\Phi[x, \vec{a}]$ holds.
    \end{defn}
    Obviously, the definition of $\MPprime{\Gamma}{\mathcal{P}}{A}$ depends on the choice of $\Phi$. This is not a problem in most cases thanks to the following proposition.

    \begin{prop}\label{prop:EquivalenceOfMPprimeAndMP}
        Assume that $\Gamma$ admits $\Gamma$-satisfaction relations and that $A$ is a set that includes all natural numbers.
        Then $\MPprime{\Gamma}{\mathcal{P}}{A} \iff \MP{\Gamma}{\mathcal{P}}{A}$.
    \end{prop}
    \begin{proof}
        Let $\Phi$ be a binary formula in $\Gamma$ witnessing that $\Gamma$ admits $\Gamma$-satisfaction relations.
        To show the forward direction, let $\phi(\vec{u})$ be in $\Gamma$ and let $\vec{a} \in A^{<\omega}$.
        Suppose that $\phi[\vec{a}]$ is $\mathcal{P}$-forceably $\mathcal{P}$-persistent.
        Then so is $\Phi[\lceil\phi\rceil, \vec{a}]$.
        Since $\lceil\phi\rceil, \vec{a} \in A$, we can apply $\MPprime{\Gamma}{\mathcal{P}}{A}$ to $\Phi[\lceil\phi\rceil, \vec{a}]$ and thus we have $\phi[\vec{a}]$.
        To show the reverse direction, fix $x$ and $\vec{a}\in A^{<\omega}$ such that $\Phi[x, \vec{a}]$ is $\mathcal{P}$-forceably $\mathcal{P}$-persistent.
        Since both $x$ and $\vec{a}$ are in $A$ by the assumption on $A$, $\MP{\Gamma}{\mathcal{P}}{A}$ can be applied to $\Phi[x, \vec{a}]$ to get $\Phi[x, \vec{a}]$.
    \end{proof}

    %We remark that one can still show \ref{prop:EquivalenceOfMPprimeAndMP} in the case that there is a simple way to code tuples from $A$ into a member of $A$, e.g.\ when $A \in \{\omega, \omega_1, \mathbb{R}\}$.

    The assumption that $\omega\subset A$ cannot be removed from Proposition \ref{prop:EquivalenceOfMPprimeAndMP}. Note that $\MPprime{\Gamma}{\mathcal{P}}{\emptyset} \iff \MPprime{\Gamma}{\mathcal{P}}{\omega}$, but $\MP{\Gamma}{\mathcal{P}}{\emptyset}$ does not imply $\MP{\Gamma}{\mathcal{P}}{\omega}$ in general:

    \begin{prop}
        $\allMP{\omega_1\text{-pres}}{\emptyset}$ is compatible with $\omega_1^L = \omega_1$. On the other hand, $\MP{\Sigma_2}{\omega_1\text{-pres}}{\omega}$ implies $\omega_1^L < \omega_1$. It follows that $\allMP{\omega_1\text{-pres}}{\emptyset} \not\Rightarrow \MP{\Sigma_2}{\omega_1\text{-pres}}{\omega}$.
    \end{prop}
    \begin{proof}
        The first part follows from the proof of Theorem \ref{thm:ConsistencyOfMPomeg1PresWithCountableParas}. To show the second part, assume toward contradiction that $\MP{\Sigma_2}{\omega_1\text{-pres}}{\omega}$ and $\omega_1^L = \omega_1$. Let $\vec{C} = \langle C_{\alpha}\mid \alpha\in\omega_1\cap\mathrm{Lim}\rangle$ be the $L$-least ladder system, i.e.\ for all $\alpha\in\omega_1 \cap\mathrm{Lim}$, $C_{\alpha}\subset\alpha$ is unbounded in $\alpha$ and has order type $\omega$.
        For each $n < \omega$, we define $f_n\colon\omega_1\to\omega_1$ by setting $f_n(\alpha)$ be the $n$-th element in the increasing enumeration of $C_{\alpha}$. As $f_n$ is a regressive function, it is constant on a stationary subset of $\omega_1$. By $\MP{\Sigma_2}{\omega_1\text{-pres}}{\omega}$, there is in fact a unique $\gamma_n$ so that $f^{-1}(\{\gamma_n\})$ is stationary and hence contains a club subset $D_n\subset\omega_1$. Then all $f_n$'s are constant on the club subset $D:=\bigcap_{n < \omega}D_n$ of $\omega_1$. However, $D$ can contain at most one point, namely $\sup_{n<\omega}\gamma_n$. This is a contradiction.
    \end{proof}
    
    It follows that $\MPprime{\Gamma}{\omega_1\text{-pres}}{\emptyset}$ does not imply $\MP{\Gamma}{\omega_1\text{-pres}}{\emptyset}$. We will show the consistency of $\allMP{\omega_1\text{-pres}}{\omega_1}$ from large cardinals in Theorem \ref{thm:ConsistencyOfMPomeg1PresWithCountableParas}.

    \subsection{Maximality Principles for provably persistent formulas}

    Next we introduce slight variants of $\mathsf{MP}$, which were actually closer to the original formulation of the Maximality Principle in \cite{StaviVaananenReflectionPrinciplesForTheContinuum} and studied in \cite{AsperoAMaximalBoundedForcingAxiom}.
    
    \begin{defn}
        Let $\phi(x)$ be a formula and let $\mathcal{P}$ be a (definable) class of posets.
        We say that $\phi(x)$ is ($\mathsf{ZFC}$-)\emph{provably $\mathcal{P}$-persistent} if $\mathsf{ZFC}$ proves
        \[
        \forall x\,(\phi[x]\to\forall\mathbb{P}\in\mathcal{P}\,(\forces_{\mathbb{P}}\phi[\check{x}])).
        \]
        If $\mathcal{P}$ is the class of all posets, we omit $\mathcal{P}$ from the terminology.
    \end{defn}
    
    \begin{defn}
        $\MPast{\Gamma}{\mathcal{P}}{A}$ is a scheme of formulas asserting that for any provably $\mathcal{P}$-persistent formula $\phi(x)$ in $\Gamma$ and any $a\in A$, if $\phi[\check{a}]$ is $\mathcal{P}$-forceable, then $\phi[a]$ holds.
    \end{defn}

    By definition, $\MP{\Gamma}{\mathcal{P}}{A}$ implies $\MPast{\Gamma}{\mathcal{P}}{A}$. Note that even in the case that $\Gamma$ is the set of $\Sigma_n$ formulas or one of $\Pi_n$ formulas, $\MPast{\Gamma}{\mathcal{P}}{A}$ should not be formulated as a first-order statement as in Subsection \ref{Subsection:FirstOrderExpressibility}. This is because in a model of $\mathsf{ZFC}+\neg\Con(\mathsf{ZFC})$, all formulas are provably $\mathcal{P}$-persistent and thus the first-order version of $\MPast{\Gamma}{\mathcal{P}}{A}$ would be always false.
    
    \begin{prop}
        Assume that $\mathcal{P}$ is an iterable $\Sigma_m$-definable class of posets.
        \begin{enumerate}
            \item $\MPast{\Pi_1}{\mathcal{P}}{A}$ holds.
            \item $\MPast{\Sigma_1}{\mathcal{P}}{A} \iff \MP{\Sigma_1}{\mathcal{P}}{A}$.
            \item For $m\leq n<\omega$, $\MPast{\Pi_n}{\mathcal{P}}{A} \iff \MP{\Pi_n}{\mathcal{P}}{A}$.
            Therefore, $\allMPast{\mathcal{P}}{A} \iff \allMP{\mathcal{P}}{A}$.
        \end{enumerate}
    \end{prop}
    \begin{proof}
        (1) follows from downward absoluteness of $\Pi_1$ formulas.
        (2) holds because $\mathcal{P}$-persistent $\Sigma_1$ formulas are always provably $\mathcal{P}$-persistent.
        To show (3), use the following observation: If $\phi(x)$ is a $\Pi_n$-formula, then the formula
        \[
        \forall\mathbb{P}\in\mathcal{P}\forces_{\mathbb{P}}\phi(\check{x})
        \]
        is also $\Pi_n$ and provably $\mathcal{P}$-persistent.
    \end{proof}

    In the next section, we will show that in many typical cases where $A\neq\emptyset$, $\MPast{\Sigma_n}{\mathcal{P}}{A}$ have strictly weaker consistency strength than $\MP{\Sigma_n}{\mathcal{P}}{A}$ for $n > 1$.

    \section{Consistency strength}

    \subsection{Maximality Principles without parameters}
    
    We first note that Maximality Principles without parameters are consistent relative to $\mathsf{ZFC}$. In \cite{HamkinsASimpleMaximalityPrinciple}, the following theorem was proved under the assumption that $\mathcal{P}$ is transfinitely iterable, but this assumption is redundant.
    
    \begin{thm}\label{thm:ConsistencyOfLightfaceMPs}
        Let $\mathcal{P}$ be an iterable class of posets. Then $\mathsf{ZFC} + \allMPempty{\mathcal{P}}$ is consistent relative to $\mathsf{ZFC}$.
    \end{thm}
    \begin{proof}
        We prove that for any finite fragment $\Delta$ of $\mathsf{ZFC} + \allMPempty{\mathcal{P}}$, some generic extension of $V$ satisfies $\Delta$.
        Let $\{\phi_k \mid k < n\}$ be a finite list of sentences such that for any $\psi$ in $\Delta\cap\allMPempty{\mathcal{P}}$, there is $k < n$ such that $\psi$ aseerts that if $\phi_k$ is $\mathcal{P}$-forceably $\mathcal{P}$-persistent, then $\phi_k$ holds.

        If $\phi_0$ is $\mathcal{P}$-forceably $\mathcal{P}$-persistent (over $V$), then let $\mathbb{P}_0 \in \mathcal{P}$ be its witness.
        Otherwise, let $\mathbb{P}_0$ be the trivial forcing poset.
        Let $G_0 \subset \mathbb{P}_0$ be $V$-generic.
        Inductively, for all $k < n-1$, if $\phi_{k+1}$ is $\mathcal{P}$-forceably $\mathcal{P}$-persistent over $V[G_0 * \cdots * G_k]$, then let $\mathbb{P}_{k+1}\in\mathcal{P}^{V[G_0 * \cdots * G_k]}$ be its witness.
        Otherwise, let $\mathbb{P}_{k+1}$ be the trivial forcing poset.
        Let $G_{k+1}\subset\mathbb{P}_{k+1}$ be $V[G_0 * \cdots * G_k]$-generic.
        
        We claim that $V[G] := V[G_0 * \cdots * G_{n-1}]$ is our desired model of $\Delta$.
        Assume that $\phi_k$ is $\mathcal{P}$-forceably $\mathcal{P}$-persistent in $V[G]$ witnessed by $\mathbb{Q}\in\mathcal{P}^{V[G]}$, since otherwise we have nothing to show.
        Since $\mathcal{P}$ is iterable, $\mathbb{P}_k * \cdots * \mathbb{P}_n * \mathbb{Q} \in \mathcal{P}^{V[G_0 * \cdots * G_{k-1}]}$.
        This poset witnesses that $\phi_k$ is $\mathcal{P}$-forceably $\mathcal{P}$-persistent in $V[G_0 * \cdots * G_{k-1}]$.
        By our choice of $\mathbb{P}_k$, $\mathbb{P}_k$ forces ``$\phi_k$ is $\mathcal{P}$-persistent'' over $V[G_0 * \cdots * G_{k-1}]$, so $\phi_k$ holds in $V[G]$.
    \end{proof}

    In particular, $\allMPempty{\text{SSP}}, \allMPempty{\omega_1\text{-pres}}$ are consistent relative to $\mathsf{ZFC}$, even though these classes of forcing posets are not transfinitely iterable.

    \subsection{Maximality Principles for transfinitely iterable classes}

    In this subsection, we often assume the following.
    
    \begin{assump}\label{assump}
    $\mathcal{P}$ and $\kappa$ are one of the following pairs of a transfinitely iterable class of posets and a cardinal:
    \begin{enumerate}
        \item $\mathcal{P}$ is the class of all posets and $\kappa = \omega_1$.
        \item $\mathcal{P} = \text{``c.c.c.''}$ and $\kappa = 2^{\omega}$.
        \item $\mathcal{P} = \text{``proper''}$ and $\kappa = \omega_2$.
        \item $\mathcal{P} = \text{``semiproper''}$ and $\kappa = \omega_2$.
        \item $\mathcal{P} = \text{``$\sigma$-closed''}$ and $\kappa = \omega_2$.
        \item $\mathcal{P} = \text{``subcomp''}$ and $\kappa = \omega_2$.
    \end{enumerate}
    \end{assump}
    In these cases, $H_{\kappa}$ is the maximal parameter set $A$ such that $\allMP{\mathcal{P}}{A}$ is consistent relative to some large cardinal axioms. See Proposition \ref{prop:MaximalParameterSets}. Also, note that $\mathcal{P}$ is $\Sigma_2$ definable.

    We start with stating Asper\'{o}'s results in \cite{AsperoAMaximalBoundedForcingAxiom}.
    A cardinal $\lambda$ is said to be $\Sigma_n$-correct (or $\Sigma_n$-reflecting) if $V_{\lambda}$ is a $\Sigma_n$-elementary submodel of $V$.

    \begin{thm}[Asper\'{o} \cite{AsperoAMaximalBoundedForcingAxiom}]\label{TheoremAspero1}
        Let $\mathcal{P}$ and $\kappa$ be as in \ref{assump}.
        Assume that there is a regular $\Sigma_n$-correct cardinal $\lambda$. Then there is $\mathbb{P}\subset V_{\lambda}$ such that $\mathbb{P}\in\mathcal{P}$, $\mathbb{P}$ forces $\MPast{\Sigma_n}{\mathcal{P}}{H_{\kappa}}$ and collapses $\lambda$ to $\kappa$.
    \end{thm}

    \begin{thm}[Asper\'{o} \cite{AsperoAMaximalBoundedForcingAxiom}]\label{TheoremAspero2}
        Let $\mathcal{P}$ and $\kappa$ be as in Assumption \ref{assump}.
        For $n\geq 1$, $\MPast{\Sigma_n}{\mathcal{P}}{A}$ implies that $\kappa$ is a regular $\Sigma_n$ correct cardinal in $L$.
    \end{thm}

    The following result seems new.

    \begin{thm}
        Let $\mathcal{P}$ and $\kappa$ be as in Assumption \ref{assump}. Let $n\geq 2$.
        \begin{enumerate}
            \item $\MP{\Pi_n}{\mathcal{P}}{H_{\kappa}}$ implies that $\kappa$ is a regular $\Sigma_{n+1}$ correct cardinal in $L$.
            \item $\MP{\Sigma_n}{\mathcal{P}}{H_{\kappa}}$ implies that $\kappa$ is a regular $\Sigma_{n+2}$ correct cardinal in $L$.
        \end{enumerate}
    \end{thm}
    \begin{proof}
    We only prove (2), because (1) is similar and easier.
    Let $a\in H_{\kappa}$ and let $\phi$ be a $\Sigma_n$ formula.
    Assume that $L\models\forall x\,\exists y\,\phi[x, y, a]$.
    The point is that the sentence
    \[
    \phi^*[a]\equiv\exists x\in L_{\kappa}\,\forall y\in L_{\kappa}\, L\models\phi[x, y, a]
    \]
    is a $\Sigma_n$ formula.
    Note that this is not provably persistent.

    We claim that $\phi^*[a]$ is $\mathcal{P}$-forceably persistent.
    One can easily find a $\mathcal{P}$-generic extension $V[g]$ where
    \[
    \phi^{**}[a]\equiv\exists x\in L_{\kappa}\,L\models \forall y\,\phi[x, y, a]
    \]
    holds by sufficiently increasing the value of $\kappa$.
    Note that $\phi^{**}$ is provably persistent and $\phi^{**}[a]$ implies $\phi^*[a]$.
    Therefore, $V[g]$ and its generic extensions satisfy $\phi^*[a]$.

    Now we can apply $\MP{\Sigma_2}{\mathcal{P}}{H_{\kappa}}$ to $\phi^*[a]$.
    Fix $x\in L_{\kappa}$ such that for all $y\in L_{\kappa}$, $L\models\phi[x, y, a]$.
    By Asper\'{o}'s result, we know that $\kappa$ is $\Sigma_n$-correct in $L$, so $L_{\kappa}\models\phi[x, y, a]$.
    Therefore, $L_{\kappa}\models\forall x\, \exists y\,\phi[x, y, a]$.
    \end{proof}

    \begin{cor}\label{cor:ExactConsistencyStrength}
        Let $\mathcal{P}$ and $\kappa$ be as in Assumption \ref{assump}. Let $n\geq 2$.
        \begin{enumerate}
            \item (Asper\'{o} \cite{AsperoAMaximalBoundedForcingAxiom}) $\MPast{\Sigma_n}{\mathcal{P}}{H_{\kappa}}$ is equiconsistent with the existence of a regular $\Sigma_n$-correct cardinal.
            \item $\MP{\Pi_n}{\mathcal{P}}{H_{\kappa}}$ is equiconsistent with the existence of a regular $\Sigma_{n+1}$-correct cardinal.
            \item $\MP{\Sigma_n}{\mathcal{P}}{H_{\kappa}}$ is equiconsistent with the existence of a regular $\Sigma_{n+2}$-correct cardinal.
        \end{enumerate}
    \end{cor}

    By this result, it immediately follows that for any $\mathcal{P}$ and $\kappa$ in Assumption \ref{assump}, $\MP{\Pi_n}{\mathcal{P}}{H_\kappa}\land\neg\MP{\Sigma_n}{\mathcal{P}}{H_\kappa}$ is consistent relative to some large cardinal axioms. However, the following question still remains open.

    \begin{que}\label{que:Pi_nMPvsSigma_nMPast}
        Let $\mathcal{P}$ and $\kappa$ be as in Assumption \ref{assump}. Is $\MP{\Pi_n}{\mathcal{P}}{H_\kappa}\land\neg\MPast{\Sigma_n}{\mathcal{P}}{H_\kappa}$ consistent for $n\geq 3$?
    \end{que}

    Let us mention another question that is possibly related to the case $n=3$ of this question. Recall that the Ground Axiom ($\mathsf{GA}$), introduced in \cite{RietzTheGroundAxiom}, is the assertion that there are no grounds of $V$ via non-trivial (set-sized) forcings. Since $\neg\mathsf{GA}$ is provably persistent $\Sigma_3$ sentence, $\mathsf{GA}$ implies $\neg\MPastempty{\Sigma_3}{\mathcal{P}}$ for any class of posets $\mathcal{P}$ containing non-trivial posets. Although it was proved in \cite{FuchinoGappoParenteGenericAbsolutenessRevisited} that $\MPastallempty{\Sigma_2}$ and $\MPallempty{\Pi_2}$ are compatible with $\mathsf{GA}$, the following question still remains open.
    
    \begin{que}\label{que:MPvsGA}
        Is $\Sigma_2$-$\mathsf{MP}$, or even $\Pi_3$-$\mathsf{MP}$, consistent with $\mathsf{GA}$?
    \end{que}

    In the rest of this subsection, we give further conclusions (of the proof) of Corollary \ref{cor:ExactConsistencyStrength}.

    \subsubsection{Maximality Principles combined with Resurrection Axioms}

    The Resurrection Axiom ($\mathsf{RA}$) was introduced by Hamkins and Johnstone in \cite{HamkinsJohnstoneresurrectionAxiomsAndUpliftingCardinals}.
   We recall the definition of its boldface version, introduced in \cite{HamkinsJohnstoneStronglyUpliftingCardinals} and reformulated by Fuchs \cite{FuchsHierarchiesOfVirtualResurrectionAxioms} in the following form.

    \begin{defn}[Hamkins--Johnstone \cite{HamkinsJohnstoneStronglyUpliftingCardinals}; Fuchs \cite{FuchsHierarchiesOfVirtualResurrectionAxioms}]
        Let $\mathcal{P}$ be a class of posets and let $\kappa\geq\omega_2$ be a cardinal.
        The \emph{boldface Resurrection Axiom for $\mathcal{P}$ at $H_{\kappa}$}, denoted $\bRA{\mathcal{P}}{H_{\kappa}}$, is the following statement: for any $\mathbb{P}\in\mathcal{P}$ and any $A\subset H_{\kappa}$, there are a $\mathbb{P}$-name $\dot{\mathbb{Q}}$ with $\forces_{\mathbb{P}}\dot{\mathbb{Q}}\in\mathcal{P}$ and a cardinal $\lambda$ such that $\forces_{\mathbb{P}*\dot{\mathbb{Q}}}$``$\lambda$ is a cardinal,'' and if $\kappa$ is regular then $\forces_{\mathbb{P}*\dot{\mathbb{Q}}}$``$\lambda$ is regular,'' such that whenever a filter $G\subset \mathbb{P}*\dot{\mathbb{Q}}$ is $V$-generic, there are $A^*\subset H_{\lambda}^{V[G]}$ and an elementary embedding
        \[
        j\colon\langle H_{\kappa}^V, \in, A\rangle \to \langle H_{\lambda}^{V[G]}, \in, A^*\rangle.
        \]
    \end{defn}

    \begin{defn}[Hamkins--Johnstone \cite{HamkinsJohnstoneStronglyUpliftingCardinals}]
        Let $\theta$ be an ordinal. A cardinal $\kappa$ is \emph{$\theta$-strongly uplifting} if for every $A\subset V_{\kappa}$ there is an inaccessible cardinal $\gamma\geq\theta$ and a set $A^* \subset V_{\gamma}$ such that
        \[
        \langle V_{\kappa}, \in, A\rangle \prec \langle V_{\gamma}, \in, A^*\rangle.
        \]
        We say that $\kappa$ is \emph{strongly uplifting} if $\kappa$ is $\theta$-uplifting for all ordinals $\theta$.
    \end{defn}

    By combining Minden's work in \cite{MindenCombiningResurrectionAndMaximality} with the proof of Corollary \ref{cor:ExactConsistencyStrength} in a straightforward way, we get the following equiconsistency results.

    \begin{thm}
        Let $\mathcal{P}$ and $\kappa$ be as in Assumption \ref{assump}.
        Let $n\geq 2$.
        \begin{enumerate}
            \item $\MPast{\Sigma_n}{\mathcal{P}}{H_{\kappa}} + \bRA{\mathcal{P}}{H_{\kappa}}$ is equiconsistent with the existence of a strongly uplifting $\Sigma_n$-correct cardinal.
            \item $\MP{\Pi_n}{\mathcal{P}}{H_{\kappa}} + \bRA{\mathcal{P}}{H_{\kappa}}$ is equiconsistent with the existence of a strongly uplifting $\Sigma_{n+1}$-correct cardinal.
            \item $\MP{\Sigma_n}{\mathcal{P}}{H_{\kappa}} + \bRA{\mathcal{P}}{H_{\kappa}}$ is equiconsistent with the existence of a strongly uplifting $\Sigma_{n+2}$-correct cardinal.
        \end{enumerate}
    \end{thm}

    \subsubsection{Recurrence Axioms}

    The Recurrence Axioms ($\mathsf{RA}$) were first introduced in Fuchino--Usuba \cite{FuchinoUsubaOnRecurrenceAxioms} and their relation with Maximality Principles for provably persistent formulas was discussed in \cite{FuchinoGappoParenteGenericAbsolutenessRevisited}.
    Based on their results, we also get the exact consistency strength of the Recurrence Axioms in typical cases.

    \begin{defn}
        A (definable) class $W$ is a \emph{ground} of $V$ if $W$ is a transitive class model of $\mathsf{ZFC}$ and $V$ is a generic extension of $W$ via a set-sized forcing, i.e., there is some forcing poset $\mathbb{P} \in W$ and a $W$-generic filter $G \subset \mathbb{P}$ such that $V = W[G]$.

        Let $\mathcal{P}$ be a definable class of forcing posets. We say that $W$ is a \emph{$\mathcal{P}$-ground} of $V$ if $W$ is a ground of $V$ witnessed by some $\mathbb{P}\in\mathcal{P}^W$.
    \end{defn}

    \begin{defn}[Fuchino--Usuba \cite{FuchinoUsubaOnRecurrenceAxioms}]
        Let $\Gamma$ be a set of formulas, let $\mathcal{P}$ be a (definable) class of posets, and $A$ be a set.
        \begin{enumerate}
            \item $\Gamma$-$\mathsf{RcA}(\mathcal{P}, A)$ is the scheme of formulas asserting that for any formula $\phi(x)$ in $\Gamma$ and any $a\in A$, if $\phi[a]$ is $\mathcal{P}$-forceable, then $\phi[a]$ holds in some ground of $V$.
            \item $\Gamma$-$\mathsf{RcA}^+(\mathcal{P}, A)$ is the scheme of formulas asserting that for any formula $\phi(x)$ in $\Gamma$ and any $a\in A$, if $\phi[a]$ is $\mathcal{P}$-forceable, then $\phi[a]$ holds in some $\mathcal{P}$-ground of $V$.
        \end{enumerate}
    \end{defn}
    
    Let us introduce few more variants of the Maximality Principles, which will not used in other subsections.

    \begin{defn}
        Let $\Gamma$ be a set of formulas, $\mathcal{P}$ be a (definable) class of posets, and $A$ be a set.
        \begin{enumerate}
            \item $\MPminus{\Gamma}{\mathcal{P}}{A}$ is the scheme of formulas asserting that for any formula $\phi(\vec{x})$ in $\Gamma$ and any $\vec{a}\in A^{<\omega}$, if $\phi[\vec{a}]$ is $\mathcal{P}$-forceably persistent, then $\phi[\vec{a}]$ holds.
            \item $\Gamma$-$\mathsf{MP}^{-, *}_{\mathcal{P}}(A)$ is the scheme of formulas asserting that for any provably persistent formula $\phi(\vec{x})$ in $\Gamma$ and any $\vec{a} \in A^{<\omega}$, if $\phi[\vec{a}]$ is $\mathcal{P}$-forceable, then $\phi[\vec{a}]$ holds.
        \end{enumerate}
    \end{defn}
    Note that $\MPminus{\Gamma}{\mathcal{P}}{A}$ has been defined in a slightly different way from \cite{FuchinoGappoParenteGenericAbsolutenessRevisited}. So, even though Lemma \ref{lemm:RcAvsMP} is essentially proved in \cite{FuchinoGappoParenteGenericAbsolutenessRevisited}, we include its proof below for reader's convenience.
    
    The easy implications between the Recurrence Axioms and the Maximality Principles can be summarized in the following diagram:
    \[
    \begin{tikzcd}
        \Gamma\text{-}\mathsf{RcA}^+(\mathcal{P}, A) \arrow[r, Rightarrow] \arrow[d, Rightarrow] & \MPast{\Gamma}{\mathcal{P}}{A} \arrow[d, Rightarrow] \\
        \Gamma\text{-}\mathsf{RcA}(\mathcal{P}, A) \arrow[r, Rightarrow]                     & \Gamma\mathchar`-\mathsf{MP}^{-, *}_{\mathcal{P}}(A)
    \end{tikzcd}
    \]
    Also, it is easy to see the following equivalence:
    \[
    \Sigma_1\mathchar`-\mathsf{RcA}^+(\mathcal{P}, A) \iff \Sigma_1\mathchar`-\mathsf{RcA}(\mathcal{P}, A) \iff \MP{\Sigma_1}{\mathcal{P}}{A}.
    \]
    
    To show the relation between the Maximality Principles and the Recurrence Axioms, we need to use the uniform definability of grounds, which was proved by Laver and Woodin. We borrow the following terminology from \cite{BagariaHamkinsTsaprounisUsubaLaverIndestructible}.

    \begin{defn}
        Let $\mathbb{P}$ be a poset and let $G\subset\mathbb{P}$ be a filter on $\mathbb{P}$. We say that \emph{$r$ succeeds in defining a ground using $\mathbb{P}$ and $G$} if letting $\delta = \lvert\mathbb{P}\rvert^+$, for every beth-fixed point $\gamma$ with $\mathrm{cf}(\gamma)>\delta$, there is a transitive $M\subset V_{\gamma}$ such that
        \begin{enumerate}
            \item $M$ satisfies $\delta$-approximation and $\delta$-cover properties to $V_{\delta}$,
            \item $M\models\mathsf{ZFC}_{\delta}$,
            \item $(\delta^+)^M = \delta^+$,
            \item $r = ({}^{{<}\delta}2)^M$,
            \item $\mathbb{P}\in M$, and
            \item $G$ is $M$-generic with $M[G] = V_{\gamma}$.
        \end{enumerate}
    \end{defn}
    Suppose that $r$ succeeds in defining a ground using $\mathbb{P}$ and $G$. By Laver's theorem \cite[Theorem 1]{LaverCertainVeryLargeCardinalsAreNotCreated}, for each beth-fixed point $\gamma$ with $\mathrm{cf}(\gamma)>\delta$, there exists a unique $M$ with the above properties (1)--(6). A class $W_r$ is defined as the unions of all such $M$. Then we have $V = W_r[G]$. In \cite{BagariaHamkinsTsaprounisUsubaLaverIndestructible}, it was observed that ``$r$ succeeds in defining a ground using $\mathbb{P}$ and $G$'' is a $\Pi_2$ statement in $r, \mathbb{P}, G$. See the paragraphs after Lemma 5 in \cite{BagariaHamkinsTsaprounisUsubaLaverIndestructible}.

    The following lemma can be found in \cite[Lemma 1.6.7]{GoodmanSigmaNCorrectForcingAxioms}.

    \begin{lemm}
        Suppose that $r$ succeeds in defining a ground. Let $a\in W_r$ and let $\phi(v)$ be a $\Sigma_n$ (resp.\ $\Pi_n$) formula, where $n\geq 2$. Then $W_r\models\phi[a]$ is expressible as a $\Sigma_n$ (resp.\ $\Pi_n$) formula with $a$ and $r$ as parameters.
    \end{lemm}

    Now it is easy to see the following.

    \begin{lemm}[Fuchino--Gappo--Parente \cite{FuchinoGappoParenteGenericAbsolutenessRevisited}]\label{lemm:RcAvsMP}
        Assume that $\mathcal{P}$ is $\Sigma_m$-definable and that $\mathcal{P}$ is an iterable class of posets. For $\max\{3, m\}\leq n<\omega$,
        \begin{enumerate}
            \item $\Sigma_n$-$\mathsf{RcA}(\mathcal{P}, A)\iff\Sigma_n$-$\mathsf{MP}^{-, *}_{\mathcal{P}}(A)$.
            \item $\Sigma_n$-$\mathsf{RcA}^+(\mathcal{P}, A)\iff\Sigma_n$-$\mathsf{MP}^*_{\mathcal{P}}(A)$.
        \end{enumerate}
    \end{lemm}
    \begin{proof}
        We only show the second part, as the proof of the first part is slightly easier. Let $\phi(v)$ be a $\Sigma_n$ formula. Consider the following formula $\phi^*(v)$: There are $r, \mathbb{P}, G$, and $\delta$ such that
        \begin{enumerate}
            \item $r$ succeeds in defining a ground using $\mathbb{P}$ and $G$ (so that $W_r$ is well-defined),
            \item $v\in W_r$,
            \item $W_r\models\mathbb{P}\in\mathcal{P}$,
            \item $W_r\models\phi[v]$.
        \end{enumerate}
        Then $\phi^*(v)$ is a $\Sigma_n$ formula, because (1) is a $\Pi_2$ statement in $r, \mathbb{P}, G$, (2) is $\Sigma_2$ (and also $\Pi_2$) in $r$, (3) is $\Sigma_{\max\{3, m\}}$ in $r, \mathbb{P}$, and (4) is $\Sigma_n$ in $r$. Note that $\phi^*(v)$ claims that there is a $\mathcal{P}$-ground $W$ such that $v\in W$ and $W\models\phi[v]$. Since $\mathcal{P}$ is iterable, $\phi^*(v)$ is provably $\mathcal{P}$-persistent. The claim immediately follows from this observation.
    \end{proof}

    \begin{thm}
    Let $\mathcal{P}$ and $\kappa$ be as in Assumption \ref{assump}. Let $n\geq 2$.
    \begin{enumerate}
        \item Both $\Sigma_n$-$\mathsf{RcA}(\mathcal{P}, H_{\kappa})$ and $\Sigma_n$-$\mathsf{RcA}^+(\mathcal{P}, H_{\kappa})$ are equiconsistent with the existence of a regular $\Sigma_n$-correct cardinal.
        \item Both $\Pi_n$-$\mathsf{RcA}(\mathcal{P}, H_{\kappa})$ and $\Pi_n$-$\mathsf{RcA}^+(\mathcal{P}, H_{\kappa})$ are equiconsistent with the existence of a regular $\Sigma_{n+1}$-correct cardinal.
    \end{enumerate}
    \end{thm}
    \begin{proof}
        For $n\geq 3$, (1) follows from \ref{cor:ExactConsistencyStrength} and \ref{lemm:RcAvsMP}.
        (2) also follows because we have the following implications:
        \[
        \begin{tikzcd}
            \Sigma_{n+1}\text{-}\mathsf{RcA}^+(\mathcal{P}, A) \arrow[r, Rightarrow] \arrow[d, Rightarrow] & \Pi_n\text{-}\mathsf{RcA}^+(\mathcal{P}, A) \arrow[r, Rightarrow] \arrow[d, Rightarrow] & \MPast{\Pi_n}{\mathcal{P}}{A} \arrow[d, Rightarrow]\\
            \Sigma_{n+1}\text{-}\mathsf{RcA}(\mathcal{P}, A) \arrow[r, Rightarrow]                     & \Pi_n\text{-}\mathsf{RcA}(\mathcal{P}, A)\arrow[r, Rightarrow] & \Pi_n\text{-}\mathsf{MP}^{-, *}_{\mathcal{P}}(A)          
        \end{tikzcd}
        \]
        Note that $\Sigma_2$-$\mathsf{RcA}(\mathcal{P}, A)$ and $\Sigma_2$-$\mathsf{RcA}^+(\mathcal{P}, A)$ are not equivalent to fragments of Maximality Principles, but still the upper bound of their consistency strength is a regular $\Sigma_2$-correct cardinal by Asper\'{o}'s proof of Theorem \ref{TheoremAspero1} in \cite{AsperoAMaximalBoundedForcingAxiom}.
    \end{proof}

    \subsection{Maximality Principles for $\omega_1$-preserving posets}\label{subsection:omega_1-pres}

    In stark contrast to the standard forcing axioms, the Maximality Principle is consistent for some highly non-iterable classes of forcings, even with parameters. Recall that the standard forcing axiom $\mathsf{FA}_{\omega_1}(\{\PP\})$ can hold for a forcing $\PP$ only if $\PP$ preserves stationary subsets of $\omega_1$ (cf.\ \cite{ForemanMagidorShelahMM}), so the forcing axiom for all $\omega_1$-preserving forcings is inconsistent. On the other hand, we will show that $\mathsf{MP}_{\omega_1\text{-pres}}(\omega_1)$ is consistent relative to some large cardinal axiom. We do not know whether this principle is \textit{practically consistent}, i.e.\ whether it provably holds in some forcing extension assuming the existence of certain large cardinals. Hence we look for other means to establish consistency.
    
    Kechris--Woodin \cite{Kechris_Woodin_2008} show that if both $\undertilde{\Sigma}^1_1$- and $\Delta^1_2$-determinacy hold then for a Turing cone of $x$, the theory of $L[x]$ is stabilized. We aim to show that $\mathsf{MP}_{\omega_1\text{-pres}}(\omega_1)$ belongs to that theory.
    
    \begin{thm}\label{thm:ConsistencyOfMPomeg1PresWithCountableParas}
        Assume that all $\Delta^1_2$- and $\undertilde{\Sigma}^1_1$-sets are determined.
        Then there is some $x_0 \in \mathbb{R}$ such that for all $x \geq_{T} x_0$,
        \[
        L[x]\models\mathsf{MP}_{\omega_1\text{-pres}}(\omega_1).
        \]
    \end{thm}
    The determinacy assumption here follows from the existence of a Woodin cardinal and a measurable cardinal above. To prove Theorem \ref{thm:ConsistencyOfMPomeg1PresWithCountableParas}, we make use of the set-version of Jensen's Coding the Universe Theorem.
    
    \begin{thm}[Jensen \cite{Beller_Jensen_Welch_1982}]\label{thm:CodingTheUniverse}
        Suppose $V=L[A]$ for a set $A$ and $\mathsf{GCH}$ holds. Then there is a cardinal-preserving set-forcing extension $V[G]$ such that $V[G]=L[y]$ for a real $y$.
    \end{thm}

    First, we prove the following variant of the result of Kechris--Woodin on the theory of $L[x]$.
    
    \begin{lemm}\label{lemm:StableTheory}
        Assume that all $\Delta^1_2$- and $\undertilde{\Sigma}^1_1$-sets are determined. Then for a Turing cone of reals $x$, we have that for all reals $y$ such that $L[y]$ is a forcing extension of $L[x]$,
        \[
        \mathrm{Th}\left(\left(L[x];\in, \alpha\mid\alpha<\omega_1^{L[x]}\right)\right)=\mathrm{Th}\left(
        \left(L[y];\in, \alpha\mid\alpha<\omega_1^{L[x]}\right)\right).
        \]
    \end{lemm}

    \begin{proof}
        For a real $x$, let $R_x$ denote the set of all reals $y$ so that $L[y]$ is a forcing extension of $L[x]$.
        Let $\varphi(x)$ be a first order formula and let 
        \[
        A_\varphi=\{x\in\mathbb R\mid \forall\alpha<\omega_1^{L[x]}\forall y\in R_x\ (L[x]\models\varphi(\alpha)\leftrightarrow L[y]\models\varphi(\alpha))\}.
        \]
        We want to show that $A_\varphi$ contains a Turing cone, so suppose this fails. By Kechris--Woodin, there is a theory $\mathcal T$  so that for some $x_0$ and all $x_0\leq_T x$, $\mathrm{Th}(L[x])=\mathcal T$. Thus the statement
        \[
        \exists\alpha<\omega_1\exists\PP\ \forces_\PP`` \exists y\in\mathbb R\ V[\dot G]=L[y]\wedge V\models\varphi(\check\alpha)\leftrightarrow L[y]\models\neg\varphi(\check\alpha))"
        \]
        belongs to $\mathcal T$. Let $\alpha_\ast$ be a term for the least $\alpha$ witnessing the above statement.
        Note that $L[x]\models \mathcal T$ for all reals $x$ with $x_0\in L[x]$, so $\alpha_\ast^x\coloneqq\alpha_\ast^{L[x]}<\omega_1^{L[x]}$ is defined for any such $x$.
        \begin{claim}
            For all $x\in\mathbb R$ with $x_0\in L[x]$ and $y\in R_x$, $\alpha_\ast^x\leq \alpha_\ast^y$.
        \end{claim}
        \begin{proof}
            Otherwise, ``there is a forcing extension of the form $L[z]$ with $\alpha_\ast^z<\alpha_\ast"$ is part of $\mathcal T$. As $\omega_1$ is indiscernible for any real, there is always a forcing witnessing this which is countable in $V$ and hence for any real $y_0$ with $x_0\in L[y_0]$ there is a real $y_1$ with $x_0\in L[y_1]$ and $\alpha_\ast^{y_0}>\alpha_\ast^{y_1}$. We can continue and build an infinite decreasing sequence of ordinals like this, contradiction.
        \end{proof}
        Let us assume wlog that $L[x_0]\models\varphi(\alpha_\ast^{x_0})$ so that $\varphi(\alpha_\ast)\in \mathcal T$ (otherwise replace $\varphi$ by $\neg\varphi$ and note that $A_\varphi=A_{\neg\varphi}$).
        \begin{claim}
            For all $x\in\mathbb R$ with $x_0\in L[x]$ and $y\in R_x$ so that $L[y]\models \neg\varphi(\alpha_\ast^x)$, we have $\alpha_\ast^x<\alpha_\ast^y$.
        \end{claim}
        \begin{proof}
            By the previous claim, $\alpha_\ast^x\leq\alpha_\ast^y$ and as $L[y]\models\neg\varphi(\alpha_\ast^x)$, $\alpha_\ast^x\neq\alpha_\ast^y$.
        \end{proof}
        Let $\xi$ be the least $x_0$-indiscernible and let $g$ be $\Col(\omega, {<}\xi)$-generic over $L[x_0]$. For $\beta<\xi$, let $y_\beta$ be a real so that $L[y_\beta]=L[x_0, g\cap \Col(\omega, {\leq}\beta)]$.
        \begin{claim}
            $\sup_{\beta<\xi}\alpha_\ast^{y_\beta}=\xi$.
        \end{claim}
        \begin{proof}
            Let $\langle\beta_i\mid i<\gamma\rangle$ be continuous and increasing so that $\beta_{i+1}$ is least such that $\alpha_\ast^{y_{\beta_{i+1}}}>\alpha_\ast^{y_{\beta_i}}$ and $\gamma$ is as large as possible. If $\gamma<\xi$ then $\gamma=\delta+1$ and $\alpha_\ast^{y_\delta}=\alpha_\ast^{y_\beta}$ for all $\delta\leq\beta<\xi$. 
            But $L[y_\delta]\models\mathcal T$, hence there is a forcing $\PP\in L[y_\delta]$ which forces the generic extension to be of the form $L[z]$ and $L[z]\models\neg\varphi(\alpha_\ast^{y_\delta})$, so $\alpha_\ast^z>\alpha_\ast^{y_\delta}$. As $L[y_\delta]$ is a forcing extension of $L[x_0]$ by a forcing of size ${<}\xi$, $\xi$ is one of the $L[y_\delta]$-indiscernibles so that we may assume $\PP\in L_\xi[y_\delta]$. Now, $\xi$ is inaccessible in $L[y_\delta]$ and hence for any sufficiently large $\varepsilon<\xi$, $L[y_\varepsilon]$ contains a generic $h$ for $\PP$ over $L[y_\delta]$. But then $L[y_\varepsilon]$ is a forcing extension of $L[y_\delta, h]$ and $\alpha_\ast^{y_\delta}<\alpha_\ast^{L[y_\delta, h]}\leq\alpha_\ast^{y_\varepsilon}$, contradiction.

            So $\gamma=\xi$ and the map $i\mapsto \alpha_\ast^{y_{\beta_i}}$ is cofinal in $\xi$.
        \end{proof}
        By Theorem \ref{thm:CodingTheUniverse}, there is a real $z$ such that $L[z]$ is a forcing extension of $L[x_0, g]$ and $\omega_1^{L[z]}=\omega_1^{L[x_0, g]}=\xi$. But $\alpha_\ast^z\geq\xi$, contradiction.

        Thus $A_\varphi$ indeed contains a Turing cone and so does the intersection $\bigcap_{\psi\in\mathrm{Fml}}A_\psi$.
        
    \end{proof}

    \begin{rem}
        Assuming slightly more determinacy, we can lift the restriction to forcing extensions in the result above. That is, there is a cone of reals $x$ such that 
        \[
        \mathrm{Th}\left(\left(L[x];\in, \alpha\mid\alpha<\omega_1^{L[x]}\right)\right)=\mathrm{Th}\left(
        \left(L[y];\in, \alpha\mid\alpha<\omega_1^{L[x]}\right)\right)
        \]
        for all $x\leq_T y$.
    \end{rem}

    \begin{proof}[Proof of Theorem \ref{thm:ConsistencyOfMPomeg1PresWithCountableParas}]
        It suffices to show that for any $x_0\in\mathbb R$, there is $x_0\leq_T x$ so that $L[x]\models\mathsf{MP}_{\omega_1\text{-pres}}(\omega_1)$.
        So let $x_0\leq_T x$ be a real such that 
        \[
        \mathrm{Th}\left(\left(L[x];\in, \alpha\mid\alpha<\omega_1^{L[x]}\right)\right)=\mathrm{Th}\left(
        \left(L[y];\in, \alpha\mid\alpha<\omega_1^{L[x]}\right)\right)
        \]
        whenever $y\in\mathbb R$ and $L[y]$ is a forcing extension of $L[x]$. 

        Suppose that $L[x, g]$ is a $\omega_1$-preserving forcing extension of $L[x]$, $\alpha<\omega_1^{L[x,g]}=\omega_1^{L[x]}$ and $\varphi(\alpha)$ holds in all further $\omega_1$-preserving forcing extensions of $L[x, g]$. As $x^\sharp$ exists, we may assume that $g\in H_{\omega_1}$.

        By Theorem \ref{thm:CodingTheUniverse}, there is a real $y$ such that $L[y]$ is an $\omega_1$-preserving forcing extension of $L[x, g]$ and hence $L[y]\models \varphi(\alpha)$. This implies $L[x]\models\varphi(\alpha)$ by our choice of $x$.
    \end{proof}
    
    So a Woodin cardinal with a measurable above is an upper bound for the consistency strength of $\mathsf{MP}_{\omega_1\text{-pres}}(\omega_1)$. We also provide a lower bound for the consistency strength of $\mathsf{MP}_{\omega_1\text{-pres}}(\omega_1)$.
    
    \begin{thm}\label{thm:RamseyFromMPomega1preserving}
        Suppose $\mathsf{MP}_{\omega_1\text{-pres}}(\omega_1)$ holds. Then there is an inner model with a Ramsey cardinal.
    \end{thm}

    To prove Theorem \ref{thm:RamseyFromMPomega1preserving}, we will make use of the following theorem on the core model $K$. Recall that $0^{\P}$ is the sharp for an inner model with a strong cardinal.

    \begin{thm}[Sharpe--Welch, \protect{\cite[Thm. 2.27$(i)$]{SharpeWelchGreatlyErdos}}]\label{thm:SharpeWelchThm}
        Suppose $0^{\P}$ does not exist. Suppose $\kappa$ is regular uncountable, $M$ is transitive and $j\colon M\rightarrow K\res \eta$ is an elementary embedding such that 
        \begin{enumerate}[label=$(\roman*)$]
            \item $\crit(j)<\kappa\leq\Ord\cap M$,
            \item $\eta = (\lambda^+)^K$ where $\lambda=j(\kappa)$ and
            \item $j[\kappa]$ is cofinal in $\lambda$.
        \end{enumerate}
        Suppose $\mathcal A\in\ran(j)$ is a $\lambda$-structure. Then there is a set $I\in K$ of good indiscernibles for $\mathcal A$ of order-type ${\geq}\kappa$.
    \end{thm}

    We refer to \cite{SharpeWelchGreatlyErdos} for the definition of a $\lambda$-structure. Simply note that to show that $\lambda$ is Ramsey in $K$, it suffices to show that any $\lambda$-structure in $K$ has admits a good set of indiscernibles of order-type ${\geq}\lambda$.

    \begin{cor}\label{cor:RamseyInK}
        Suppose $0^{\P}$ does not exist. Suppose $\kappa$ is regular uncountable, $M$ is transitive and $j\colon M\rightarrow K$ is a $\Sigma_1$-elementary embedding such that 
        \begin{enumerate}[label=$(\roman*)$]
            \item $\kappa\in M$,
            \item $\crit(j)<\kappa$,
            \item $j(\kappa)=\kappa$ and $(\kappa^+)^K\in\ran(j)$.
        \end{enumerate}
        Then $K\models``\kappa\text{ is Ramsey}."$
    \end{cor}
    \begin{proof}
        Let $\eta=(\kappa^+)^K$ and $j(\bar\eta)=\eta$. Then 
        \[
        j\res N\colon N\rightarrow K\res\eta
        \]
        is fully elementary where $N=M\res\bar\eta$. If $\kappa$ is not Ramsey in $K$ then the $K$-least $\kappa$-structure $\mathcal A$ without a good set of indiscernibles of order-type ${\geq}\kappa$ in $K$. So $\mathcal A$ is definable in $K$ from the parameter $\kappa\in \ran(j)$, hence $\mathcal A\in\ran(j)$. This contradicts Theorem \ref{thm:SharpeWelchThm}.
    \end{proof}

    For the remainder of this section, let $R$ denote the set of regressive functions $f\colon\omega_1\to\omega_1$. 
    
    \begin{defn}
        The \emph{regressive oracle} is the partial function 
        \[
        \mathcal O\colon R \rightharpoonup \omega_1
        \]
        which associates, if possible, to a regressive $f\in R$ the unique $\gamma$ such that $f^{-1}(\gamma)$ contains a club.
    \end{defn}
    
    This function serves as an oracle in the sense that it is external to the model we are going to apply it to, the core model.
    
    \begin{defn}
    Suppose $X\subseteq K$. Then the \emph{$\Sigma_n$-oracle hull} $\Sigma_n$-$\mathrm{Hull}^K_{\mathcal O}(X)$ is the smallest $\Sigma_n$-elementary substructure $Z\prec_{\Sigma_n}K$ with 
        \begin{enumerate}[label=$(\roman*)$]
            \item $X\subseteq Z$ and 
            \item $\mathcal O[Z\cap\dom(\mathcal O)]\subseteq Z$.
        \end{enumerate}
    \end{defn}
    
    As $K$ has definable Skolem terms, it is easy to see that the $\Sigma_n$-oracle hull always exists. We will show that it is of a particularly simple form.
    \begin{lemm}\label{lemm:OracleHullPresentation}
        Suppose $X\subseteq K$. Then 
        \[
        \Sigma_n\text{-}\mathrm{Hull}^K_\mathcal O(X)=\Sigma_n\text{-}\mathrm{Hull}^K(X\cup Y)
        \]
        where $Y=\mathcal O[\Sigma_n\text{-}\mathrm{Hull}^K(X)\cap\dom(\mathcal O)]$.
    \end{lemm}

    \begin{proof}
        The inclusion ``$\supseteq"$ is clear, so we will show ``$\subseteq$". It suffices to demonstrate that $\Sigma_n$-$\mathrm{Hull}^K(X\cup Y)$ is closed under $\mathcal O$. Suppose that $f\in \dom(\mathcal O)\cap \Sigma_n$-$\mathrm{Hull}^K(X\cup Y)$, say 
        \[
        f=\tau^K(p, \mathcal O(g_0),\dots, \mathcal O(g_k))
        \]
        where $p, g_0,\dots, g_k\in\Sigma_n$-$\mathrm{Hull}^K(X)$ and $\tau$ is a $\Sigma_n$-term.  Consider the function $h\in R$ defined by 
        \[h(\alpha)=
        \begin{cases}
            \tau(p, g_0(\alpha),\dots, g_k(\alpha))(\alpha) & \text{if this is defined and ${<}\alpha$}\\
            0 & \text{ else.}
        \end{cases}
        \]
        
        Clearly, $h\in \Sigma_n$-$\Hull^K(X)$. On a club $C_i$, we have $g_i(\alpha)=\mathcal O(g)$ and hence on the club $\bigcap_{i\leq k} C_i$, we have 
        \[
        h(\alpha)=\tau^K(p, g_0(\alpha),\dots, g_k(\alpha))(\alpha)=\tau^K(p,\mathcal O(g_0),\dots,\mathcal O(g_k))(\alpha)=f(\alpha)
        \]
        so that $h\in\dom(\mathcal O)$ and $\mathcal O(g)=\mathcal O(h)\in Y$.
    \end{proof}

    \begin{proof}[Proof of Theorem \ref{thm:RamseyFromMPomega1preserving}]
        We may assume that $0^\P$ does not exist as otherwise there is an inner model with a strong cardinal.
        The only consequence of $\mathsf{MP}_{\omega_1\text{-pres}}(\omega_1)$ we will use is that stationary costationary subsets of $\omega_1$ are not easily definable.
        \begin{claim}\label{claim:ManyStationaryAreClub}
            Suppose $S\subseteq\omega_1$ is a stationary set and $S$ is uniformly definable in all $\omega_1$-preserving forcing extensions using countable ordinals as parameters. Then $S$ contains a club.
        \end{claim}
        \begin{proof}
            Apply $\mathsf{MP}_{\omega_1\text{-pres}}(\omega_1)$ to the statement ``$S$ contains a club".
        \end{proof}
         It follows that any $f\in R$ which is uniformly definable in all $\omega_1$-preserving forcing extensions using countable ordinals as parameters is in the domain of $\mathcal O$.

         We will now construct a sequence $\langle \alpha_i\mid i<\omega_1\rangle$ of countable ordinals. If $\alpha_j$ is defined for all $j<i$ then set 
         \[
         X_i\coloneqq \Sigma_1\text{-}\Hull^K_\mathcal O(\{\alpha_j\mid j<i\}\cup\{\omega_1, \lambda\})
         \]
         where $\lambda=((\omega_1^V)^+)^K$.
         Finally, let $\alpha_i$ be the 2nd ordinal in 
         \[
         \bigcap\{f^{-1}(\mathcal O(f))\mid f\in X_i\cap\dom(\mathcal O)\}.
         \]
         
         \begin{claim}\label{claim:DomOracleContainsXi}
             For all $i<\omega_1$, $R\cap X_i\subseteq\dom(\mathcal O)$.
         \end{claim}
         \begin{proof}
             We prove this by induction on $i$. Observe that the regressive oracle $\mathcal O$ is almost unchanged when passing to a $\omega_1$-preserving forcing extension, at most the domain increases. Hence the definition of $\alpha_j$ for $j<i$ is uniform in all $\omega_1$-preserving forcing extensions and it follows that all elements of $X_i$ are uniformly definable in all $\omega_1$-preserving forcing extensions using countable ordinals as parameters. Thus $R\cap X_i\subseteq\dom(\mathcal O)$. 
         \end{proof}

         \begin{claim}\label{claim:alphaiAddsNoSmallOrdinals}
             For all $i<\omega_1$, $X_{i+1}\cap\alpha_i=X_i\cap\alpha_i$.
         \end{claim}
         \begin{proof}
             Suppose $\gamma\in X_{i+1}\cap\alpha_1$. We can write $\gamma$ as 
             \[
             \gamma=\tau^K(p, \mathcal O(g_0),\dots, \mathcal O(g_k),\alpha_i)
             \]
             for some $p,g_0,\dots, g_k\in X_i$ and some $\Sigma_1$-term $\tau$ by Lemma \ref{lemm:OracleHullPresentation}. Consider the function $h\in R$ defined via 
            \[h(\beta)=
            \begin{cases}
                \tau^K(p, g_0(\beta),\dots, g_k(\beta),\beta) & \text{if this is defined and ${<}\beta$}\\
                0 & \text{ else.}
            \end{cases}
            \]
            and note that $h\in X_i$. The definition of $\alpha_i$ implies that $g_j(\alpha_i)=\mathcal O(g_j)$ for $j\leq k$. Further, $h\in\dom(\mathcal H)$ by Claim \ref{claim:DomOracleContainsXi} so that $h(\alpha_i)=\mathcal O(h)$ and hence
            \[\gamma=\tau^K(p,g_0(\alpha_i),\dots,g_k(\alpha_i),\alpha_i)=h(\alpha_i)=\mathcal O(h)\in X_i.
            \]
         \end{proof}
         Now let $X_\ast=\bigcup_{i<\omega_1} X_i\prec_{\Sigma_1} K$ and let $j\colon M\rightarrow X_\ast$ be the anticollapse. Clearly, $\omega_1\in M$ and as $\omega_1\in X_\ast$ and $\vert X_\ast\cap\omega_1\vert=\omega_1$, we have $j(\omega_1)=\omega_1$. Further, as we chose the $\alpha_i$ to be the 2nd ordinal in the relevant intersection and by Claim \ref{claim:alphaiAddsNoSmallOrdinals}, we made sure that $\omega_1\nsubseteq X_\ast$ so that $\crit(j)<\omega_1$. Applying Corollary \ref{cor:RamseyInK}, we find that $\omega_1$ is Ramsey in $K$.
    \end{proof}

    The proof shows that an inner model with a Ramsey cardinal exists assuming the conclusion of Claim \ref{claim:ManyStationaryAreClub}. A stronger principle is ``every $\Sigma_2$-definable (without parameters) stationary subset of $\omega_1$ contains a club". This implies that the club filter restricted to $\mathrm{HOD}$ is an ultrafilter, so $\omega_1$ is measurable in $\mathrm{HOD}$ (otherwise consider the $\mathrm{HOD}$-least stationary co-stationary subset). So in fact this is equivalent to the statement
    \begin{center}
        ``all ordinal definable stationary subsets of $\omega_1$ contain a club.''
    \end{center}
    Hoffelner \cite{hoffelneruniverseordinaldefinablestationarycostationary} showed that the consistency strength of this principle is exactly a measurable cardinal.

    We end this subsection with several remaining questions.

    \begin{que}\label{que:omega_1-pres}\leavevmode
        \begin{enumerate}
            \item What is the exact consistency strength of $\mathsf{MP}_{\omega_1\text{-pres}}(\omega_1)$?
            \item Is $\mathsf{MP}_{\omega_1\text{-pres}}(\mathbb{R})$ (or even $\mathsf{MP}_{\omega_1\text{-pres}}(\mathbb{R}\cup\omega_2)$) consistent?
        \end{enumerate}
    \end{que}
    
    Note that $\mathsf{MP}_{\omega_1\text{-pres}}(H_{\omega_2})$ is inconsistent as it implies that every stationary set contains a club. The argument in \cite{SchindlerBMMandStrongCardinals} shows that $\Sigma_1$-$\mathsf{MP}_{\omega_1\text{-pres}}(\mathbb{R})$ implies the existence of an inner model with a strong cardinal.

    \begin{que}
        Is $\mathsf{MP}_{\text{c-pres}}(A)$, where ``c-pres'' denotes the class of cardinal-preserving posets, consistent for some parameter set $A\supseteq\omega$?
    \end{que}

    Our proof of Theorem \ref{thm:ConsistencyOfMPomeg1PresWithCountableParas} does not give a positive answer to this question, because $\mathsf{GCH}$ cannot always be forced by cardinal-preserving posets, and thus \ref{thm:CodingTheUniverse} may not be applicable.

    \section{Separating $\Sigma_n$ and $\Pi_n$-Maximality Principles}

    The consistency strength of $\SnMPallR$ is strictly higher than the consistency strength of $\PnMPallR$. Moreover, the natural way to force $\MPallempty{\Sigma_n}$ seems to be to force $\MPastall{\Sigma_{n+2}}{\mathbb{R}}$ which implies $\PnMPallR$, in fact even $\MPastall{\Pi_{n+1}}{\mathbb{R}}$. So it might be tempting to conjecture that $\PnMPallR$ is a consequence of $\SnMPallR$. However, we will show that this is not the case.
    
    \begin{thm}\label{thm:SeparateMaximalityPrinciples}
        Suppose the existence of a regular reflecting cardinal is consistent with $\mathsf{ZFC}$. Then for any $n\geq 2$, there is a model of $\mathsf{ZFC}$ in which $\SnMPallR\wedge\neg\PnMPallempty$ holds.
    \end{thm}

    We start with a rough sketch of the argument. The basic idea is to start with a proper class of independent buttons $\langle B_\alpha\mid\alpha\in\Ord\rangle$ all of which are initially unpushed. To be more precise, each $B_\alpha$ is a statement (in parameter $\alpha$) which fails in $V$, is provably persistent for all forcings and for any class $A\subseteq\Ord$, there is a forcing extension $V[G]$ in which $B_i$ holds (``is pushed") iff $\alpha\in A$.

    We will carefully choose a $\Sigma_n$-definable class $A$ and go to $V[G]$ in which only $B_i$ for $i\in A$ is pushed. The regular $\Sigma_{n+2}$-correct cardinal $\kappa$ survives this step, so we may pass to a further generic extension $V[G][h]$ in which $\SnMPallR$ holds. Of course, $\PnMPallR$ may hold as well in $V[G][h]$. The trick is to ``unpush" one of the buttons $B_i$ with $i\in A$ by passing to a ground $W$ of $V[G][h]$ so that the $\Pi_n$-sentence
    \[
    \forall\alpha\in A\ B_\alpha
    \]
    fails in $W$, but holds and does so persistently, in $V[G][h]$. Here we use that $A$ is still $\Sigma_n$-definable in $W$ and this is one of the reasons we start with $V=L$. This will be enough to ensure that $\PnMPallempty$ fails in $W$. The difficult part here is to arrange that $\SnMPallR$ remains true. This will come down to our choice of $A$ and $i$. First, let $\lambda_0$ be sufficiently large so that all $\Sigma_{n}$ statements about reals have a witness in $V_{\lambda_0}^{V[G][h]}$, i.e. if $\varphi(v_0, v_1)$ is $\Pi_{n-1}$, $x$ is a real and $\exists y\varphi(y, x)$ holds then $\exists y\in V_{\lambda_0}^{V[G][h]}\ \varphi(y, x)$. We will then be able to find a $\Sigma_{n-1}$-correct $\lambda\geq\lambda_0$ and some $i\geq\lambda$ as well as a ground $W$ of $V[G][h]$ so that $V_\lambda^W=V_\lambda^{V[G][h]}$, $B_i$ is unpushed in $W$ and, crucially, $\lambda$ remains $\Sigma_{n-1}$-correct in $W$.  As $\lambda\geq\lambda_0$, $W$ and $V$ agree about $\Sigma_n$-statements with real parameters as these are upwards absolute from $V_{\lambda}^W$ to $W$. 

    The only thing to worry about now is that there may be a $\Sigma_{n}$-statement about a real $x$ which can be forced to be persistent over $W$, but not over $V[G][h]$. But it is easy to see that this is not possible as $W$ is a ground for set forcing of $V[G][h]$. It follows that $\SnMPallR$ holds in $W$.

    Arranging for such a $\lambda$ to exist will give use the most headaches. It is much easier in the case $n=2$ as $H_\lambda\prec_{\Sigma_1}V$ for any uncountable cardinal $\lambda$ and in this case we can get away with $A=\Ord$. For $n>2$, the argument above will only work with a very careful selection of $A$.

    The exact choice of buttons is of minor relevance, in our case we will start with $V=L$ and $B_\alpha$ will be the statement 
    \[
    \aleph_{\alpha+1}^L\text{ contains a fresh subset},
    \]
    that is, there is some $a\subseteq\aleph_{\alpha+1}^L$ with $a\notin L$, yet $a\cap\alpha\in L$ for $\alpha<\aleph_{\alpha+1}^L$. For any definable class $A$, the Easton support product 
    \[
    \PP_A\coloneqq \prod_{\alpha\in A}\Add(\omega_{\alpha+1}, 1)
    \]
    is a definable class forcing which pushes the buttons $B_\alpha$ with $\alpha\in A$.

    \begin{prop}
        $\PP_A$ pushes only the buttons $B_\alpha$ with $\alpha\in A$.
    \end{prop}
    \begin{proof}
        For $\alpha\notin A$, note that we can split $\PP_A$ into $\PP_{<\alpha}\times\PP_{>\alpha}$ with $\PP_{>\alpha}$ being ${<}\omega_{\alpha+2}$-closed and $\PP_{<\alpha}$ of size $<\omega_{\alpha+1}$. Finally, recall that a forcing of size $<\kappa$ cannot add a fresh subset to $\kappa$.
    \end{proof}

    \subsection{Finding the class $A$}
    
    Let us fix some $n\geq 1$. Our aim is to show that assuming $V=L$ there is a $\Sigma_{n+1}$-definable $A$ with a number of nice properties.
    \begin{thm}\label{thm:NiceAexists}
        Assume $V=L$. Then there is a $\Sigma_{n+1}$-definable proper class $A\subseteq\Ord$ so that 
        \begin{enumerate}
            \item Whenever $\lambda$ is $\Sigma_n$-correct, $A\cap\lambda=A^{L_\lambda}$ and 
            \[
            (L_\lambda;\in, A\cap\lambda)\prec_{\Sigma_n}(L;\in, A).
            \]
            \item Whenever $\lambda$ is $\Sigma_n$-correct, $B\subseteq \Ord, B\cap \lambda= A\cap \lambda$ and $B\triangle A$ is a set, we have 
            \[
            (L_\lambda;\in, A\cap\lambda)\prec_{\Sigma_n}(L;\in, B).
            \]
        \end{enumerate}
    \end{thm}
    Point $(2)$ above roughly states that there is no $\Sigma_n$-way to detect set-sized changes in $A$ from below. If $i\geq\lambda$ and $i\in A$ then necessarily 
    \[
    (L_\lambda;\in, A\cap\lambda)\not\prec_{\Sigma_{n+1}}(L;\in, A\setminus\{i\})
    \]
    as the statement ``$\exists j\in A\ j\notin A\setminus\{i\}"$ holds true and is $\Sigma_{n+1}$ over $(L;\in, A\setminus\{i\})$, so $(2)$ is the most we can ask for.

    While we will never pass to a forcing extension in this section, our technique to find $A$ will make use of the class forcing method.

    \begin{defn}
        Suppose $\PP$ is a class forcing. A filter $G\subseteq\PP$ is \emph{$\Sigma_n$-generic} if $\PP\cap D\neq\emptyset$ for all dense $D\subseteq\PP$ which are $\Sigma_n$-definable (with parameters) over $(V;\in, \PP, \leq, \perp)$.
    \end{defn}

    The only class forcing of interest to us in this section is ``Cohen forcing at $\Ord$.''
    
    \begin{defn}
        The class forcing $\Addinf$ consist of functions $p\colon\alpha\rightarrow 2$ with $\alpha\in\Ord$, ordered by reverse inclusion.
    \end{defn}
    
    The forcing $\Addinf$ is ${<}\Ord$-closed so it is tame and does not add any sets. Also, we remark that $\Addinf$, $\leq$ and $\perp$ are all $\Sigma_0$-definable.

    \begin{lemm}\label{lemm:DefinableGenericExists}
        Assume $V=L$. There is a $\Sigma_{n+1}$-definable class $G\subseteq\Addinf$ so that for all $\Sigma_n$-correct $\lambda$ the following holds.
        \begin{enumerate}
            \item $G^{L_\lambda}=G\cap L_\lambda$.
            \item If $D\subseteq\Addinf$ is $\Sigma_n$-definable from parameters in $L_\lambda$ and any $p\in G\cap L_\lambda$ can be extended to a condition in $D$ then $D\cap G\cap L_\lambda\neq\emptyset$.
        \end{enumerate}
    \end{lemm}
    Note that $(2)$ implies that $G$ is $\Sigma_n$-generic over $L$ as well as over $L_\lambda$ for any $\Sigma_n$-correct $L_\lambda$.
    \begin{proof}
        Already the existence of any $\Sigma_{n+1}$-definable $\Sigma_n$-generic is not entirely trivial: for a $\Sigma_n$-class $D\subseteq\Addinf$, the statement ``$D$ is dense" is $\Pi_{n+1}$, but we want $A$ to be $\Sigma_{n+1}$. Furthermore, it may happen that $\lambda$ is $\Sigma_n$ correct and $L_\lambda$ and $L$ disagree about whether $D$ is dense. The trick to get around this issue is to never ask the question whether a given class is dense, instead only ask whether a given $p$ can be extended to a given $\Sigma_n$ class $D$ which is only $\Sigma_n$.

        Using the canonical wellorder $<_L$ of $L$, let $\langle \tau_\alpha\mid\alpha\in\Ord\rangle$ be a $\Sigma_1$-definable enumeration of all terms for $\Sigma_n$-classes.  

        Now define a decreasing sequence $\langle p_\alpha\mid\alpha\in\Ord\rangle$ through $\Addinf$ as follows: Let $p_0=\emptyset$ and for $\alpha\in\mathrm{Lim}$, $p_\alpha=\bigcup_{\beta<\alpha} p_\beta$. If $p_\alpha$ is defined then, if possible, let $p_{\alpha+1}$ be the $<_L$-least $q\leq p_\alpha$ with $q$ in the class defined by the term $\tau_\alpha$. Otherwise, $p_{\alpha+1}=p_\alpha$. 

        Note that the definition in the successor step is of complexity $\Sigma_n\vee\Pi_n$ as there is a $\Sigma_n$-formula defining $\Sigma_n$-truth, so the whole recursion is $\Delta_{n+1}$. Moreover, carrying out this recursion in some $\Sigma_n$-correct $L_\lambda$ yields $\langle p_\alpha\mid\alpha<\lambda\rangle$. Finally, it is easy to see that 
        \[
        G=\{q\in\Addinf\mid \exists\alpha\ p_\alpha\leq q\}
        \]
        has the property $(1)$. For $(2)$, let $D$ be any $\Sigma_n$-definable class (using parameters in $L_\lambda$) so that any $p\in G\cap L_\lambda$ can be extended to a condition in $D$. Then $D$ is the evaluation of a term $\tau_\alpha$ with $\alpha<\lambda$ and hence $p_{\alpha+1}\in D\cap G\cap L_{\lambda}$.
    \end{proof}
    We remark that the $A$ defined above is even $\Delta_{n+1}$. This is best possible in the sense that a $\Sigma_n$-generic $A$ for $\Addinf$ can be neither $\Sigma_n$ nor $\Pi_n$-definable. 

    We will go on and prove that $A$ above witnesses Theorem \ref{thm:NiceAexists} to hold. To do so, we will prove a localized version of the forcing theorem for $\Addinf$ along the following lines.

    \begin{defn}
        We enrich the forcing language for $\Addinf$ with an additional atomic formula
     $v\in\dot A$ and define $p\forces\check x\in\dot A$ via $x\in\dom(p)\wedge p(x)=1$. 
   \end{defn}

    \begin{lemm}\label{lemm:LocalForcingTheoremForAddInf}
        Suppose $G$ is $\Sigma_n$-generic for $\Addinf$ with associated generic class $A\subseteq\Ord$. For any $\Sigma_n$-formula $\varphi(v_0,\dots, v_k)$ and $x_0,\dots, x_k\in V$, the following are equivalent
        \begin{enumerate}
            \item $(V;\in, A)\models \varphi(x_0,\dots, x_k)$.
            \item There is $p\in G$ with $p\forces\varphi(\check x_0,\dots, \check x_k)$.
        \end{enumerate}
    \end{lemm}
    
    While this is true, it is itself not very useful. The issue is that the forcing relation for $\Addinf$ seems to be too complicated: For a $\Sigma_n$ formula $\varphi$, the naive definition of the relation $p\forces\varphi(\check x_0,\dots,\check x_k)$ is not better than $\Pi_{n+2}$, while we would really like it to be $\Sigma_n$. Through other means, we will see that it is $\Pi_{n+1}$ for $n\geq 1$, but we do not know whether it is $\Sigma_n$. One problem here, which is specific to class forcing, is that $\Addinf$ has antichains which are proper classes and so $p\forces\exists x\varphi(x)$ is not in general equivalent to $\exists \dot x\ p\forces\varphi(\dot x)$, i.e. mixing fails.  Even for $\varphi\in\Sigma_0$, it is a priori not clear whether the forcing relation for $\varphi$ is any better than $\Pi_2$ as a similar issue arises for bounded existential quantification.

    To deal with this, we replace the forcing relation $\forces$ with a stronger yet less complicated relation $\forces^+$ for which we can nonetheless prove a localized forcing theorem. This in turn will rely on a careful analysis of the complexity of the relation $\forces^+$.  This analysis holds more generally for all ${<}\Ord$-distributive forcings, so we carry it out in more generality.

    %We will take a close look at the forcing relation for $\Addinf$. As $\Addinf$ does not add new sets, there is no reason to deal with any names except for check names\footnote{Observe that $p\forces\forall x\varphi(x)$ is equivalent to $\forall x\ p\forces\varphi(\check x)$ and $p\forces\exists x\varphi(x)$ is equivalent to $\forall q\leq p\exists r\leq q\, x\ r\forces\varphi(\check x)$.}. 

    \begin{defn}
        A first order formula $\varphi$ is \emph{reduced} if 
        \begin{itemize}
            \item no implication appears in $\varphi$ and
            \item negations appear only immediately in front of atomic subformulas.
        \end{itemize}
    \end{defn}

    Note that any formula is tautologically equivalent to a reduced formula: First replace all implications $\varphi\rightarrow\psi$ by $\neg\varphi\vee\psi$, then push all negations through as far as possible and finally eliminate double negations.

    \begin{defn}
        Let $\PP$ be a class forcing.
        For reduced formulas $\varphi(v_0,\dots, v_k)$ in the language $\mathcal L_{\dot G}=\{\in,\dot G\}$ and $x_0,\dots, x_k\in V$, $p\in\PP$, we define $p\forces^+\varphi(\check x_0,\dots, \check x_k)$ by induction of $\varphi$ as follows:
        \begin{itemize}
            \item $p\forces^+ x= y$ iff $x=y$.
            \item $p\forces^+ x\neq y$ iff $x\neq y$.
            \item $p\forces^+ x\in y$ iff $x\in y$.
            \item $p\forces^+ x\notin y$ iff $x\notin y$.
            \item $p\forces^+ x\in \dot G$ iff $x\in \PP\wedge p\leq x$.
            \item $p\forces^+ x\notin \dot G$ iff $x\notin \PP\vee (x\in \PP\wedge p\perp x)$.
            \item $p\forces^+ \exists z\in y\ \varphi(z,  x_0,\dots, x_k)$ iff $\exists z\in y\ p\forces^+\varphi( z,  x_0,\dots,  x_k)$.
            \item $p\forces^+\forall z\in y\ \varphi(z,  x_0,\dots,  x_k)$ iff $\forall z\in y\ p\forces^+\varphi( z, x_0,\dots,  x_k)$.
            \item $p\forces^+\varphi_0( x_0,\dots,  x_k) \wedge \varphi_1( x_0,\dots,  x_k)$ iff $p\forces^+\varphi_0( x_0,\dots, x_k)\wedge p\forces^+\varphi_1( x_0,\dots, x_k)$.
            \item $p\forces^+\varphi_0( x_0,\dots,  x_k) \vee \varphi_1( x_0,\dots,  x_k)$ iff $p\forces^+\varphi_0( x_0,\dots, x_k)\vee p\forces^+\varphi_1( x_0,\dots, x_k)$.
            \item $p\forces^+ \exists y\varphi(y,  x_0,\dots,  x_k)$ iff $\exists y\ p\forces^+\varphi( y, x_0,\dots,  x_k)$.
            \item $p\forces^+\forall y\varphi(y, x_0,\dots,  x_k)$ iff $\forall q\leq p\forall y\exists r\leq q\ r\forces^+\varphi( y, x_0,\dots, x_k)$.
        \end{itemize}
    \end{defn}
    We will now prove some basic properties of $\forces^+$.
    \begin{lemm}\label{lemm:BasicPropertiesOfStrongRelation}
        Let $\PP$ be some ${<}\Ord$-distributive class forcing. The strong forcing relation $\forces^+$ has the following properties for any reduced formula $\varphi$ in the language $\mathcal L_{\dot G}$.
        \begin{enumerate}
            \item If $p\forces^+\varphi( x_0,\dots,  x_k)$ and $q\leq p$ then $q\forces^+\varphi( x_0,\dots,  x_k)$.
            \item $p\forces^+\varphi( x_0,\dots, x_k)$ implies $p\forces\varphi(\check x_0,\dots,\check x_k)$.
            \item $p\forces\varphi(\check x_0,\dots,\check x_k)$ iff
            \[
            \{q\leq p\mid q\forces^+\varphi( x_0,\dots, x_k)\}
            \]
            is dense below $p$.
            \item The class 
            \[
            \{p\in\PP\mid p\forces^+\varphi( x_0,\dots,  x_k)\vee p\forces^+\varphi^\neg( x_0,\dots, x_k)\}
            \]
            is dense. Here, $\varphi^\neg$ is the reduced formula induced by $\neg\varphi$.
        \end{enumerate}
    \end{lemm}
    \begin{proof}
        $(1)$ is easy and a straightforward induction on the complexity of $\varphi$ proves $(2)$, so we skip it. 
        
        Note that the reverse implication in $(3)$ is already a consequence of $(2)$: If 
        \[
        \{q\leq p\mid  q\forces^+\varphi( x_0,\dots, x_k)\}
        \] is dense below $p$ then so is $\{q\leq p\mid q\forces\varphi(\check x_0,\dots,\check x_k)\}$ and hence $p\forces\varphi(\check x_0,\dots,\check x_k)$.
        
        We will now prove the forwards implication of $(3)$. We proceed by induction on the complexity of $\varphi$. The atomic case is easy.
        
        \begin{description}
            \item[\underline{$\varphi=\psi_0\wedge\psi_1$}] Let $q\leq p\forces\psi_0\wedge\psi_1$. By induction there is $r\leq q$ with $r\forces^+\psi_0$ and further $s\leq r$ with $s\forces^+\psi_1$. By $(1)$, $s\forces^+\psi_0$ as well and hence $s\forces^+\psi_0\wedge\psi_1$.
            \item[\underline{$\varphi=\psi_0\vee\psi_1$}] If $p\forces\psi_0\vee\psi_1$ then for any $q\leq p$ there is $r\leq q$ and $i<2$ so that $r\forces\psi_i$. By induction, there is $s\leq r$ so that $s\forces^+\psi_i$ and hence $s\forces^+\psi_0\vee\psi_1$.
            \item[\underline{$\varphi=\exists z\in y\ \psi(z)$}] If $p\forces\exists z\in\check y\ \varphi(z)$ then for any $q\leq p$, there is $r\leq q$ and $z\in y$ with $r\forces\psi(\check z)$.  By induction, there is $s\leq r$ with $s\forces^+\psi(z)$ and hence $s\forces^+\exists z\in y\ \psi(z)$.
            \item[\underline{$\varphi=\forall z\in y\ \psi(z)$}]  Suppose that $p\forces\forall z\in\check y\ \psi(z)$. Then clearly for all $z\in y$, $p\forces\psi(\check z)$ and hence 
            \[
            D_z=\{q\leq p\mid q\forces^+\psi( z)\}
            \]
            is dense open below $p$. As $\PP$ is ${<}\Ord$-distributive, $\bigcap_{z\in y} D_z$ is dense below $p$ and for any $q\in\bigcap_{z\in y} D_z$, we have $q\forces^+\forall z\in y\ \psi(z)$.
            \item[\underline{$\varphi=\exists z \psi(z)$}] Similar to the bounded existential case.
            \item[\underline{$\varphi=\forall z \psi(z)$}] Assuming $p\forces\forall z\psi(z)$, we will show that in fact $p\forces^+\forall z\varphi(z)$. Given $q\leq p$ and $y$, we have $q\forces\varphi(\check y)$ so that by induction, there is $r\leq q$ with $r\forces^+\varphi(y)$. Hence $p\forces^+\forall y\psi(y)$ by definition of $\forces^+$.
        \end{description}
        For $(4),$ suppose there is no $q\leq p$ with $q\forces^+\varphi$. Then by $(3)$, there is no $q\leq p$ with $q\forces\varphi$ and hence $p\forces\neg\varphi$. As $\neg\varphi$ and $\varphi^\neg$ are tautologically equivalent, $p\forces\varphi^\neg$ so that by $(3)$ again, there is some $q\leq p$ with $q\forces^+\varphi^\neg$.
    \end{proof}
    
    The strong forcing relation $\forces^+$ has some quirks which are not shared by the usual forcing relation. For example, it can happen that $p\forces^+\forall x\varphi(x)$, yet there is no single $x$ with $p\forces^+\varphi(x)$. Further, there may be reduced sentences $\varphi,\psi$ so that $\varphi$ tautologically implies $\psi$, $p\forces^+\varphi$ and yet $p\not\forces^+\psi$.

    \begin{prop}\label{prop:ComplexityOfStrongRelation}
        Suppose $\PP$ is a class forcing and $\varphi(v_0,\dots,v_k)$ is a reduced $\Sigma_m$ (resp. $\Pi_m$) formula in the language $\mathcal L_{\dot G}$. Then the relation 
        \[
        \{(p, x_0,\dots, x_k)\mid p\forces^+\varphi( x_0,\dots, x_k)\}
        \]
        is $\Sigma_m$ (resp. $\Pi_m$) definable over $(V;\in, \PP,\leq,\perp)$. 
    \end{prop}
    \begin{proof}
        We prove this by induction on the complexity of $\varphi$.
        \begin{description}
            \item[\underline{$\varphi\in\Sigma_0$}] This is straightforward.

            \item[\underline{$\varphi\in\Pi_1$}] Say $\varphi(a)=\forall x\psi(x, a)$ where $\psi$ is $\Sigma_0$. For a given $x$, let 
            $$D_x=\{q\in\PP \mid\forall r\in\mathrm{tc}(\{x, a\})\ (q\leq r\vee q\perp r)\}.$$
            An inspection of the definition of $\forces^+$ shows that if $q\in D_x$ and $r\leq q$ with $r\forces^+\psi(x, a)$ then already $q\forces^+\psi(x, a)$. Further, as $\PP$ is ${<}\Ord$-distributive, $D_x$ is dense. It follows that $p\forces^+\varphi(a)$ is equivalent to 
            $$\forall q\leq p\forall x (q\in D_x\rightarrow q\forces^+\psi(x, a))$$
            which is $\Pi_1$.
            \item[\underline{$\varphi\in\Sigma_{m+1}$}] Suppose $\varphi=\exists x\psi(x, a)$ with $\psi\in\Pi_n$. We have $p\forces^+\varphi(a)$ iff $\exists x\ p\forces^+\psi(x, a)$ which is $\Sigma_{m+1}$ by induction.
            \item[\underline{$\varphi\in\Pi_{m+2}$}] $\varphi$ is of the form $\forall x_0\dots\forall x_k \psi(x_0,\dots,x_k)$ with $\psi$ a reduced $\Sigma_n$ formula. With a moment of reflection, one can see that
        \[
        p\forces^+\forall x_0\dots\forall x_k \psi(x_0,\dots,x_k)
        \]
        is equivalent to 
        \[
        \forall q\leq p\forall x_0\dots\forall x_k\exists r\leq q\ r\forces^+\psi(x_0,\dots,x_k)
        \]
        which is $\Pi_{m+1}$ by induction as $m+1>0$. 
        \end{description}

    \end{proof}
    In the proof of Lemma \ref{lemm:BasicPropertiesOfStrongRelation}, we showed that if $\varphi$ is of the form $\forall x\psi(x)$ then $p\forces\varphi(\check x)$ iff $p\forces^+\varphi(x)$. Thus we also get the optimal complexity of the standard forcing relation for the $\Pi_m$ complexity classes. 
    \begin{cor}
        Suppose $\PP$ is a ${<}\Ord$-distributive class forcing, $m\geq 1$ and $\varphi$ is a $\Pi_m$-formula in the language $\mathcal L_{\dot G}$. Then the relation 
        \[
        \{(p, x_0,\dots, x_k)\mid p\forces\varphi(\check x_0,\dots, \check x_k)\}
        \]
        is $\Pi_m$-definable over $(V;\in, \PP,\leq,\perp)$.
    \end{cor}

    We are now in good shape to prove the localized forcing theorem for the strong forcing relation.

    \begin{lemm}\label{lemm:LocalStrongForcingTheoremForAddInf}
        Let $\PP$ be a ${<}\Ord$-distributive class forcing and suppose $G$ is $\Sigma_n$-generic for $\PP$. For any reduced $\Sigma_n$-formula $\varphi(v_0,\dots, v_k)$ in the language $\mathcal L_{\dot G}$ and $x_0,\dots, x_k\in V$, the following are equivalent
        \begin{enumerate}
            \item $(V;\in, G)\models \varphi(x_0,\dots, x_k)$.
            \item There is $p\in G$ with $p\forces^+\varphi(x_0,\dots, x_k)$.
        \end{enumerate}
    \end{lemm}
    \begin{proof}
        We do so by induction on the complexity of $\varphi$. The case where $\varphi$ is a Boolean combination of atomic formulas as well as the bounded existential quantification case is easy. We will further skip the cases of conjunction and disjunction.
        \begin{description}
            \item[\underline{$\varphi=\forall x\in y\ \psi(x)$}] If there is $p\in G$ with $p\forces^+\forall x\in  y\psi(x)$ then $\forall x\in y\ p\forces^+\psi(x)$ so that $\forall x\in y (V;\in, G)\models \psi(x)$ and hence $(V;\in, G)\models\forall x\in y\ \psi(x)$. On the other hand, if $(V;\in, G)\models \forall x\in y\ \psi(x)$ then for all $x\in y$ there is some $p_x\in G$ with $p\forces^+\psi( x)$. The set 
            \[
            D=\{q\in\PP\mid \forall x\in y\ q\leq p_x\vee \exists x\in y\ q\perp p_x\}
            \]
            is dense in $\PP$ as $\PP$ is ${<}\Ord$-distributive. Also, $D$ is $\Sigma_0$-definable and hence $G$ meets $D$. It follows that there is $p\in G$ with $\forall x\in y\ p\forces^+\psi(x)$ so that $p\forces^+\forall x\in y\ \psi(x)$.
            \item[\underline{$\varphi=\exists x\ \varphi(x)$}] If $(V;\in, G)\models\exists x\psi(x)$ then let $x$ witness this. Hence there is $p\in G$ with $p\forces^+\psi(x)$ so that $p\forces^+\exists x\psi(x)$. On the other hand, if $p\forces^+\exists x\ \psi(x)$ and $p\in G$ then there is $x$ so that $p\forces^+\psi( x)$ and hence $(V;\in, G)\models\psi(x)$.
            \item[\underline{$\varphi=\forall x\psi(x)$}] We may assume that $\psi$ is at most $\Sigma_{n-2}$. Suppose first that 
            \[
            (V;\in, G)\models\forall x\ \psi(x).
            \]
            Let $\psi^\neg$ denote the reduced $\Pi_{n-2}$ formula induced by $\neg\psi$. The set 
            \[
            D\coloneqq \{p\in\PP\mid p\forces^+\forall x\ \psi(x)\vee p\forces^+\exists x\psi^\neg(x)\}
            \]
            is dense by Lemma \ref{lemm:BasicPropertiesOfStrongRelation} and $\Sigma_{n-1}\vee\Pi_{n-1}$-definable. Thus  
            there is some $p\in D\cap G$ as $G$ is $\Sigma_n$-generic. We cannot have that $p\forces^+\psi(x)^\neg$ for some $x$ as then $(V;\in, G)\models\neg\psi(x)$, so $p\forces^+\forall x\psi(x)$ is the only other possibility.

            Now assume that $p\forces^+\forall x\psi(x)$ for some $p\in G$. This implies that for any $x$, the set 
            \[
            D_x=\{q\in \PP\mid q\perp p\vee (q\leq p\wedge q\forces^+\psi(x))\}
            \]
            is dense and, as $D_x$ is $\Sigma_{n-2}$, $D_x\cap G\neq\emptyset$. Hence $(V;\in, G)\models \psi(x)$ by induction, for any $x$.
        \end{description}
    \end{proof}

    \begin{rem}
        The proof shows that the assumption ``$G$ is $\Sigma_n$-generic" can be weakened slightly to ``$G$ is $\Sigma_{n-1}\vee\Pi_{n-1}$-generic" in the obvious sense.
    \end{rem}

    While we will make no use of Lemma \ref{lemm:LocalForcingTheoremForAddInf}, we remark that it is a consequence of Lemma \ref{lemm:LocalStrongForcingTheoremForAddInf} and Lemma \ref{lemm:BasicPropertiesOfStrongRelation}. 

    \begin{proof}[Proof of Theorem \ref{thm:NiceAexists}]
        Let $G$ be the $\Sigma_n$-generic filter for $\Addinf$ given by Lemma \ref{lemm:DefinableGenericExists} and let $A\subseteq\Ord$ denote the associated $\Sigma_n$-generic class of ordinals. Suppose that $\lambda$ is $\Sigma_n$-correct. Then $A\cap\lambda=A^{L_\lambda}$ as $G\cap L_\lambda=G^{L_\lambda}$.
        \begin{claim}
            Suppose $B\subseteq\Ord$ is $\Sigma_n$-generic for $\Addinf$ with $B\cap\lambda=A\cap\lambda$. Then 
            \[
            (L_\lambda;\in, A\cap\lambda)\prec_{\Sigma_n}(L;\in, B).
            \]
        \end{claim}
        \begin{proof}
            Let $H\subseteq\Addinf$ be the filter associated to $B$ and let $\varphi$ be a reduced $\Sigma_n$-formula and $x_0,\dots, x_k\in L_\lambda$. It suffices prove $(L;\in, H)\models\varphi(x_0,\dots, x_k)$ implies $(L_\lambda;\in, G\cap L_\lambda)\models\varphi(x_0,\dots, x_k)$, so let us assume the former. By Lemma \ref{lemm:LocalStrongForcingTheoremForAddInf},
            \[
                \exists p\in H\ p\forces^+\varphi( x_0,\dots,  x_k)
            \]
            so $H$ meets the $\Sigma_n$-definable class $D=\{q\in\Addinf\mid q\forces^+\varphi( x_0,\dots, x_k)\}$. By the property $(2)$ of $G$ guaranteed by Lemma \ref{lemm:DefinableGenericExists}, $G\cap L_\lambda\cap D\neq\emptyset$, so 
            \[
            (L_\lambda;\in, G\cap L_\lambda)\models\varphi(x_0,\dots,x_k)
            \]
            by Lemma \ref{lemm:LocalStrongForcingTheoremForAddInf}.
        \end{proof}
        The theorem follows as any set-sized change to $A$ results in another $\Sigma_n$-generic filter for $\Addinf$.
    \end{proof}

    \subsection{Separating Maximality Principles}

    We now aim to proof the main theorem. For the remainder of this section, \textbf{we assume $V=L$ and that $\kappa$ is a reflecting cardinal}. 
    
    \begin{prop}\label{prop:PPAhasNoProperClassAntichains}
        Let $A\subseteq\Ord$ be any definable class. There is no definable proper class antichain in $\PP_A$. 
    \end{prop}
    The large cardinal assumption is necessary to prove this. If there are no inaccessible cardinals and there is a definable global wellorder then $\PP_A$ has a proper class antichain for all proper classes $A$.
    \begin{proof}
        Without loss of generality we may assume that $A$ is definable with parameters in $L_\kappa$ and suppose towards a contradiction that there is a proper class definable antichain $\mathcal A$. Similarly, we may assume that $\mathcal A$ is definable with parameters in $L_\kappa$. Using the canonical wellorder of $L$, let $\langle p_\alpha\mid\alpha\in\Ord\rangle$ be a definable one-to-one enumeration of $\mathcal A$. Let
        \[
        f\colon I\rightarrow V,\ f(\lambda)=p_\lambda\res\lambda,
        \]
        where $I$ is the class of inaccessible cardinals and note that $f$ is essentially regressive in the sense that $f(\lambda)\in V_\lambda$.
        \begin{claim}
             There is a proper class $S$ so that $f$ is constant on $S$.
        \end{claim}
        \begin{proof}
            If not then for all $q$, $f^{-1}(\{q\})$ is bounded, say by $\alpha_q$. Note that
            \[
            \forall q\in L_\kappa\ \alpha_q<\kappa
            \]
            as $L_\kappa\prec L$ and $q\mapsto\alpha_q$ is definable using parameters in $L_\kappa$. On the other hand, we must have $\alpha_q\geq\kappa$ for $q=f(\kappa)$, contradiction.
        \end{proof}
        Let $\lambda_0\in S$ and take $\lambda_1\in S$ sufficiently large so that $\dom(p_{\lambda_0})\subseteq\lambda_1$. As $\lambda_0,\lambda_1\in S$, we have $p_{\lambda_0}\res\lambda_0=p_{\lambda_1}\res\lambda_1$. Clearly, $p_{\lambda_0},p_{\lambda_1}$ are compatible then, contradiction.
    \end{proof}

    The important consequence of this for us is that we can make use of mixing, i.e.\ if $p\forces_{\PP_A}\exists x\varphi(x)$ then there is a name $\dot x$ so that $p\forces_{\PP_A}\varphi(\dot x)$. We could not make use of this in the prior section as $\Addinf$ has proper class antichains.

    \begin{lemm}
        Suppose $A$ is a definable class. For $n\geq 2$ and $\Gamma$ either $\Sigma_n$ or $\Pi_n$, for any $\varphi\in\Gamma$ the relation 
        \[
        \{(p, \dot x_0,\dots, \dot x_k)\mid p\forces_{\PP_A}\varphi(\dot x_0,\dots, \dot x_k)\}
        \]
        is $\Gamma$-definable over $(V;\in, A)$.        
    \end{lemm}
    \begin{proof}
        Let us start with the base case $n=2$. The main point is that if $\lambda$ is inaccessible, $G$ is generic for $\PP_A$ then $G_\lambda=G\cap V_\lambda$ is generic over $V$ for the set forcing $\PP_{A\cap\lambda}$, $\lambda$ remains inaccessible and $V^{V[G]}_\lambda=V_\lambda[G_\lambda]$. In particular, $V_\lambda[G_\lambda]\prec_{\Sigma_1} V[G]$. 
        
        \begin{claim}
            For a $\Sigma_2$-formula $\varphi$, we have
             \[
             p\forces_{\PP_A}\varphi(\dot x)\Leftrightarrow\exists\lambda\ \lambda\text{ is inaccessible}\wedge p, \dot x\in V_\lambda\wedge p\forces_{\PP_{A\cap\lambda}}``V_{\check\lambda}^{V[\dot G]}\models\varphi(\dot x)".
             \]
        \end{claim}
        \begin{proof}
            We only prove the forwards direction. If $p\forces_{\PP_A}\varphi(\dot x)$ then we also have 
            \[
            p\forces_{\PP_A}\exists\mu\ \mu\text{ is inaccessible and } V_\mu^{V[\dot G]}\models \varphi(\dot x).
            \]
            As there are no proper class antichains, there is some $\lambda_0$ so that 
            \[
            p\forces_{\PP_A}\exists\mu\leq\check\lambda_0\ \mu\text{ is inaccessible and } V_\mu^{V[\dot G]}\models \varphi(\dot x).
            \]
            Finally, for any inaccessible $\lambda\geq\lambda_0$, we then have 
            \[
            p\forces_{\PP_{A\cap\lambda}}V_{\check\lambda}^{V[\dot G]}\models\varphi(\dot x).
            \]
        \end{proof}
        
        As $\PP_{A\cap\lambda}$ is a set-forcing for any $\lambda$, the above claim provides a $\Sigma_2$-definition of $p\forces_{\PP_A}\varphi(\dot x)$ for any $\Sigma_2$-formula $\varphi$.
        
        Similarly, if $\varphi$ is $\Pi_2$ then 
        \[
        p\forces_{\PP_A}\varphi(\dot x)\Leftrightarrow\forall\lambda\ (\lambda\text{ is inaccessible}\wedge p, \dot x\in V_\lambda)\rightarrow p\forces_{\PP_{A\cap\lambda}}``V_{\check\lambda}^{V[\dot G]}\models\varphi(\dot x)"
        \]
        is $\Pi_2$.

        We can now proceed by induction. If the claim holds for a $\Pi_n$ formula $\varphi$ then 
        \[
        p\forces\exists x\varphi(x,\dot y)\Leftrightarrow\exists \dot x\in V^{\PP_A}\ p\forces\varphi(\dot x,\dot y)
        \]
        is $\Sigma_{n+1}$. Here we used Proposition \ref{prop:PPAhasNoProperClassAntichains} to use mixing. The remaining case is easier.
    \end{proof}

    \begin{proof}[Proof of Theorem \ref{thm:SeparateMaximalityPrinciples}]
        We follow the sketch laid out in the introduction, filling out the details. Let $A$ be the $\Sigma_n$-definable class given by Theorem \ref{thm:NiceAexists} (by plugging in $n-1$) and  let $G$ be generic for $\PP_A$ and let $G_\kappa=G\cap L_\kappa$.
        
        \begin{claim}
            $L_\kappa[G_\kappa]\prec L[G]$.
        \end{claim}
        \begin{proof}
            As $A$ is definable, $(L_\kappa;\in, A\cap \kappa)\prec (L;\in, A)$ and hence $(L_\kappa;\in, \PP_A\cap L_\kappa)\prec(L;\in, \PP_A)$. If $D\subseteq\PP_A\cap L_\kappa$ is any dense subset then 
            \[
            D'\coloneqq\{p\in\PP_A\mid p\res\kappa\in D\}
            \]
            is dense in $\PP_A$ as $\PP_A$ has Easton support. This easily implies that $G\cap L_\kappa$ is generic over $L_\kappa$. So if $L_\kappa[G_\kappa]\models\varphi(\dot x^G)$ for some $\dot x\in L_\kappa$ then there is $p\in G$ so that $p\res\kappa\forces\varphi(\dot x)$. Hence $L[G]\models\varphi(\dot x^G)$ as well.
        \end{proof}
        
        In particular, $\kappa$ is regular and $\Sigma_{n+2}$-correct in $L[G]$, so we may pass to a forcing extensions $L[G][h]$ for a forcing of size $\kappa$ so that $L[G][h]\models\SnMPallR$. Let $\lambda_0$ be sufficiently large so that $V_{\lambda_0}^{L[G][h]}$ contains a witness for any true $\Sigma_n$-statement in a real parameter. Now let $\lambda\geq\lambda_0$ be regular $\Sigma_{2n+2}$-correct. Let $i>\lambda$, $i\in A$, let $G_-$ be $G$ with the generic for $\Add(\omega_{i+1}, 1)$ deleted and $G_\lambda=G\cap L_\lambda$.
    
        \begin{claim}
            $L_\lambda[G_\lambda]\prec_{\Sigma_{n-1}}L[G_-]$.
        \end{claim}
        \begin{proof}
            In any case, we have that $G_\lambda$ is generic over $L_\lambda$ for $\PP_A^{L_\lambda}=\PP_A\cap L_\lambda=\PP_{A\cap\lambda}$. Also recall that 
            \[
            (L_\lambda;\in, A\cap\lambda)\prec_{\Sigma_{n-1}}(L;\in, A\setminus\{i\}).
            \]
            First start with the easy case $n-1=1$. We have that $\lambda$ is inaccessible and hence $L_\lambda[G_\lambda]=H_{\lambda}^{L[G_-]}\prec_{\Sigma_1}L[G_-]$, so we are done.
        
            Now assume $n-1\geq 2$. If $\varphi$ is $\Sigma_{n-1}$ and 
            \[
            L[G_-]\models\varphi(\dot x^{G_-})
            \]
            for some name $\dot x\in L_\lambda$ then there is $p\in G_-$   with 
            \[
            (L;\in, A\setminus\{i\})\models p\forces_{\PP_{A\setminus\{i\}}}\varphi(\dot x).
            \]
            We will show that already 
            \[
            (L_\lambda;\in, A\cap\lambda)\models p\res\lambda\forces_{\PP_{A\cap\lambda}}\varphi(\dot x).
            \]
            As $\lambda$ is $\Sigma_{2n+2}$-correct in $L$, 
            \[(L_\lambda;\in, A\cap\lambda)\models \{q\in\PP_{A\cap\lambda}\mid q\forces\varphi(\dot x)\vee q\forces_{\PP_{A\cap\lambda}}\neg\varphi(\dot x)\}\text{ is dense}.
            \]
            Further, for any $q\leq p\res\lambda$, we have that the $\Sigma_{n-1}$-statement
            \[
            (L;\in, A\setminus\{i\})\models\exists r\leq q\ r\forces{\PP_{A\setminus\{i\}}}\varphi(\dot x)
            \]
            reflects down to $(L_\lambda;\in, A\cap\lambda)$. 
        
            Hence $L_\lambda[G_\lambda]\models\varphi(\dot x^{G_\lambda})$ as $p\res\lambda\in G_\lambda$.
        \end{proof}
        
        It is easier to see that $\lambda$ is $\Sigma_{n-1}$-correct in $L[G]$.
        
        \begin{claim}
            $L[G_-][h]\models\SnMPallR$.
        \end{claim}
        \begin{proof}
            Suppose $\varphi(x)$ is $\Sigma_n$, $x$ a real and $\QQ\in L[G_-][h]$ forces that $\varphi(x)$ holds in all further forcing extensions. By taking $H$ to be $\QQ$-generic over $L[G][h]$, we see that $L[G][h][H]\models\varphi(x)$ as this is a forcing extension of $L[G_-][h][H]$.
        
            By choice of $\lambda$, we have $L_\lambda[G_\lambda][h]\models\varphi(x)$. As $\Sigma_{n-1}$-correct cardinals are preserved by small forcing, $L_\lambda[G_\lambda][h]\prec_{\Sigma_{n-1}} L[G_-][h]$. As $\Sigma_n$-formulas are upwards absolute from $L_\lambda[G_\lambda][h]$ to $L[G_-][h]$, we have $L[G_-][h]\models\varphi(x)$.
        \end{proof}
    
        \begin{claim}
            In $L[G_-][h]$, there is no fresh subset of $\omega_{i+1}^L$.
        \end{claim}
        \begin{proof}
            $L[G_-][h]$ arises as an extension of $L$ by first a $<\omega_{i+2}^L$-closed class forcing and then a further extension by a forcing of size ${<}\omega_{i+1}$. Thus $L[G_-][h]$ cannot contain a fresh subset of $\omega_{i+1}^L$.
        \end{proof}
    
        Finally, consider the statement 
        \[
        \forall\alpha\in A^L\ (\omega_{\alpha+1}^L\text{ contains a fresh subset}).
        \]
        As $\Sigma_n$-truth over $L$ is $\Sigma_n$-definable, $A^L$ is still $\Sigma_n$-definable and as $n\geq 2$, this is a $\Pi_n$-statement. It holds in any forcing extension of $L[G][h]$, but not in $L[G_-][h]$ and hence $\PnMPallempty$ fails in $L[G_-][h]$. This completes the proof of Theorem \ref{thm:SeparateMaximalityPrinciples}.
    \end{proof}

    We raise two questions regarding possible generalizations of Theorem \ref{thm:SeparateMaximalityPrinciples}.

    \begin{que}
        For which $n\geq 3$ is $\SnMPallR \land \neg\PnMPallempty$ consistent with a proper class of supercompact/extendible cardinals?
    \end{que}

    Our proof of Theorem \ref{thm:SeparateMaximalityPrinciples} shows the consistency of $\MP{\Sigma_n}{\mathcal{P}}{\mathbb{R}}\land\neg\MPempty{\Pi_n}{\mathcal{P}}$, where $\mathcal{P}$ is the class of $\kappa$-closed forcings for any uncountable cardinal $\kappa$.

    \begin{que}
        Is $\MP{\Sigma_n}{\mathcal{P}}{\mathbb{R}}\land \neg\MPempty{\Pi_n}{\mathcal{P}}$ consistent for other iterable classes $\mathcal{P}$ of posets, such as c.c.c.\ or proper?
    \end{que}

    \subsection{Yet another version of Maximality Principles}

    The following variant of Maximality Principles appears in \cite{HamkinsASimpleMaximalityPrinciple}.

    \begin{defn}
        Let $\Gamma$ be a set of formulas, $\mathcal{P}$ be a class of posets, and $A$ be a set. Then $\MPplus{\Gamma}{\mathcal{P}}{A}$ is the scheme of formulas asserting that for any formula $\phi(\vec{x})$ in $\Gamma$ and any $\vec{a}\in A^{<\omega}$, if $\phi[\vec{a}]$ is $\mathcal{P}$-forceably $\mathcal{P}$-persistent, then $\phi[\vec{a}]$ holds and is $\mathcal{P}$-persistent.
    \end{defn}

    By definition, $\MPplus{\Gamma}{\mathcal{P}}{A}$ implies $\MP{\Gamma}{\mathcal{P}}{A}$. Other easy facts on $\MPplus{\Gamma}{\mathcal{P}}{A}$ are summarized as follows.

    \begin{prop}
        Assume that $\mathcal{P}$ is an iterable $\Sigma_m$-definable class of posets.
        \begin{enumerate}
            \item $\MPplus{\Pi_1}{\mathcal{P}}{A}$ holds.
            \item $\MPplus{\Sigma_1}{\mathcal{P}}{A} \iff \MP{\Sigma_1}{\mathcal{P}}{A}$.
            \item For $m\leq n<\omega$,
            \begin{align*}
                \MPplus{\Pi_n}{\mathcal{P}}{A} & \iff \MP{\Pi_n}{\mathcal{P}}{A},\\
                \MP{\Pi_n}{\mathcal{P}}{A} & \Longrightarrow \MPplus{\Sigma_{n-1}}{\mathcal{P}}{A}.
            \end{align*}
            Therefore, $\allMPplus{\mathcal{P}}{A} \iff \allMP{\mathcal{P}}{A}$.
        \end{enumerate}
    \end{prop}

    Our proof of \ref{thm:SeparateMaximalityPrinciples} implies the following.

    \begin{thm}
        Suppose the existence of a regular reflecting cardinal is consistent with $\mathsf{ZFC}$. Then for any $n\geq 2$, there is a model of $\mathsf{ZFC}$ in which $\MPall{\Sigma_n}{\mathbb{R}}\land\neg\MPplusallempty{\Sigma_n}$ holds.
    \end{thm}
    \begin{proof}
        The model produced in the proof of Theorem \ref{thm:SeparateMaximalityPrinciples} is the desired model. Consider the $\Sigma_n$-sentence ``if $\omega_1 \in A$ then $\omega_1$ has a fresh subset.''
    \end{proof}

    \section{The impact of large cardinals and standard forcing axioms}

    We used $V=L$ extensively in the prior sections, so it is natural to ask whether we can accommodate large cardinals. This is definitely impossible for $n=2$.

    \begin{thm}[Woodin]\label{thm:WoodinsStatTowerForcing}
        Suppose there is a proper class of Woodin cardinals. Then $\MPall{\Pi_2}{\mathbb{R}}$ holds.
    \end{thm}
    \begin{proof}
        Suppose $\varphi$ is a $\Pi_2$-formula, $x\in\mathbb R$ and $\PP$ forces that $\varphi(x)$ holds in all further extensions, yet $\varphi(x)$ fails in $V$. Then there is an inaccessible $\kappa$ so that $V_\kappa\models\neg\varphi(x)$. Let $\delta>\max\{\kappa,\vert\PP\vert\}$ be a Woodin cardinal and let $\mathbb{P}_{<\delta}$ be the stationary tower forcing at $\delta$. We can find $V$-generic $G \subset \mathbb{P}_{<\delta}$ so that $V[G]$ contains an extension of $V$ by $\PP$, in particular $V[G]\models\varphi(x)$. Let $j\colon V\rightarrow M$ be the induced generic elementary embedding. Then $j(\delta)=\delta$, so $j(\kappa)<\delta$. By elementarity of $j$, $j(\kappa)$ is inaccessible in $M$ and $V_{j(\kappa)}^M\models\neg\varphi(x)$. But $V_\delta^M =V_\delta^{V[G]}$ so $j(\kappa)$ is inaccessible in $V$ and $V_{j(\kappa)}^{V[G]}\models\neg\varphi(x)$. This implies $V[G]\models\neg\varphi(x)$, contradiction.
    \end{proof}

    What happens for $n\geq 3$ is really unclear. As we noted before Question \ref{que:MPvsGA}, the Ground Axiom ($\mathsf{GA}$) implies $\neg\MPastempty{\Sigma_3}{\mathcal{P}}$ for any non-trivial class of posets $\mathcal{P}$. Since Reitz \cite{RietzTheGroundAxiom} showed that natural large cardinal assumptions are compatible with $\mathsf{GA}$, no natural large cardinal assumption implies $\MPastallempty{\Sigma_3}$ outright. We will extend this idea to show that natural large cardinals do not prove $\MPastallempty{\Sigma_2}$.

    \begin{defn}[Reitz \cite{RietzTheGroundAxiom}]
        We say that a set of ordinals $A$ is coded into the continuum function at $\beta$ iff 
        \[
        \forall\alpha<\sup A\ (2^{\aleph_{\beta+\alpha+1}}=\aleph_{\beta+\alpha+2}\iff\beta\in A).
        \]
        The \emph{Continuum Coding Axiom} ($\CCA$) states that every set of ordinals is coded into the continuum function at some $\beta$.
    \end{defn}

    $\CCA$ implies $\mathsf{GA}$. Note that $\neg\CCA$ is a provably persistent $\Sigma_3$ statement, but unfortunately it is not $\Sigma_2$. However, we can turn $\CCA$ into a stronger $\Pi_2$-sentence by requiring that the coding happens sufficiently local.

    \begin{defn}
        The \emph{Local Continuum Coding Axiom} ($\LCCA$) asserts that for all cardinals $\lambda$ with $\lambda=\aleph_\lambda$, we have $V_\lambda\models\CCA$.
    \end{defn}

    If $V_\lambda\models\CCA$ for an ordinal $\lambda$ then $\lambda=\aleph_\lambda$, so $\LCCA$ is the ``fastest" coding one could ask for. A slight variant of $\LCCA$ was also used in \cite{GoldbergUsubasTheoremIsOptimal} under the name of the Local Ground Axiom. Since $\neg\LCCA$ is a provably persistent $\Sigma_2$ statement, $\MPastallempty{\Sigma_2}$ implies $\neg\LCCA$.

    \begin{lemm}\label{lemm:LCCAConsistent}
        For any model of $\mathsf{ZFC}$, there is a class forcing extension in which $\LCCA$ holds.
    \end{lemm}
    \begin{proof}
        We may assume that we start with a model of $\mathsf{GBC}+\mathsf{GCH}$, in particular we have access to a class global wellorder $\prec$. Define $\xi_\alpha$ for $\alpha\geq\omega$ by $\xi_\omega=0$, $\xi_{\alpha+1}=\xi_\alpha + \alpha^{++}$ and $\xi_\alpha=\sup_{\beta<\alpha}\xi_\beta$ for $\beta\in\mathrm{Lim}$. 
       
        Define an Easton support iteration $\PP=\langle \PP_\alpha,\Q_\alpha\mid\alpha\in\Ord\rangle$ as follows: 
        For $\alpha\in\Ord$, $\Q_\alpha$ is the name for the trivial forcing unless $\alpha$ is a limit ordinal. In this case, let $\Q_\alpha$ be a $\PP_\alpha$-name for the forcing $\QQ_\alpha$ which codes every subset of $\alpha$ in $V^{\PP_\alpha}$ at an ordinal in the interval $[\xi_\alpha,\xi_{\alpha+1})$. The definition of $\xi_{\alpha+1}$ is chosen so that this interval is sufficiently long for the coding take place: we will inductively have that $(2^\alpha)^{V^{\PP_\alpha}}\leq (\alpha^{++})^V$. More precisely, $\QQ_\alpha$ is the Easton support product
        \[
        \prod_{i<2^{\vert\alpha\vert}}\prod_{j\in x_i}\mathbb{A}(\aleph_{\xi_{\alpha}+\alpha\cdot i + j + 1})
        \]
        where $\mathbb A(\lambda)=\Add(\lambda,\lambda^{++})$ and $\langle x_i\mid i<2^{\vert\alpha\vert}\rangle$ is the evaluation of the $\prec$-least $\PP_\alpha$-name for an enumeration of $\mathcal P(\alpha)^{V^{\PP_\alpha}}$. Note that $\QQ_\alpha$ does not add any new subsets of $\alpha$.

        In $V^\PP$, every subset of some limit ordinal $\alpha$ is coded into the continuum function at an ordinal ${<}\xi_{\alpha+1}$. It is easy to see that if $\lambda=\aleph_\lambda$ then $\lambda=\xi_\lambda=\aleph_{\xi_\lambda}$ and hence $V_{\lambda}^{V^\PP}\models\CCA$.
    \end{proof}

    This construction preserves natural large cardinals, such as measurable, Woodin, supercompact, extendible cardinals, etc. So $\LCCA$ is ``consistent with all natural large cardinals."\footnote{A precise statement which follows from Lemma \ref{lemm:LCCAConsistent} is that assuming the existence of a proper class of Woodin cardinals, $\LCCA$ is $\Omega$-consistent in terms of Woodin's $\Omega$-logic (cf.\ \cite{WoodinPmax}).} Thus Woodin's Theorem \ref{thm:WoodinsStatTowerForcing} is best possible in the sense that no (consistent) large cardinal assumption proves $\MPastempty{\Sigma_2}{\mathcal{P}}$ for any provably non-trivial class of forcings $\mathcal P$.

    \begin{cor}
        Suppose that the existence of a proper class of Woodin cardinals is consistent. Then there is a model of $\mathsf{ZFC}+\MPall{\Pi_2}{\mathbb{R}}\land\neg \MPastallempty{\Sigma_2}$.
    \end{cor}
    \begin{proof}
        Any model of $\mathsf{ZFC}$ with a proper class of Woodin cardinals extends to a model $M$ of $\LCCA$ with a proper class of Woodin cardinals. As there is a proper class of Woodin cardinals, $M\models\MPall{\Pi_2}{\mathbb{R}}$ and as $\LCCA$ holds in $M$, $M\models \neg\MPastallempty{\Sigma_2}$.
    \end{proof}

    We can use this idea to solve a question of Ikegami--Trang, who proved the following results.

    \begin{thm}[Ikegami--Trang \cite{IkegamiTrangOnAClassOfMaximilityPrinciples}]\leavevmode
        \begin{enumerate}
            \item $\mathsf{MM}^{++}$ does not imply $\MP{\Pi_2}{\mathrm{SSP}}{H_{\omega_2}}$.
            \item Suppose $\mathsf{MM}^{++}$ holds and there is a proper class of Woodin cardinals. Then $\MP{\Pi_2}{\mathrm{SSP}}{H_{\omega_2}}$ holds.
        \end{enumerate}
    \end{thm}

    They also showed that for some $n\geq 3$, the conclusion of (2) cannot be strengthened to $\MPempty{\Sigma_n}{\mathrm{SSP}}$. Consequently, \cite[Question 5.5]{IkegamiTrangOnAClassOfMaximilityPrinciples} asks whether the theory $\mathsf{MM}^{++}\,+$``there is a proper class of Woodin cardinals" proves $\MPempty{\Sigma_2}{\mathrm{SSP}}$. We will show that the answer is no in a strong sense.

    \begin{thm}\label{thm:MM_plus_Failure_of_Sigma_2-MP}
        Suppose the theory $\mathsf{MM}^{++}+``$there is a proper class of Woodin cardinals" is consistent. Then so is the theory $\mathsf{MM}^{++}+``$there is a proper class of Woodin cardinals"$+\neg\MPastempty{\Sigma_2}{\mathrm{SSP}}$.
    \end{thm}
    \begin{proof}
        It was shown in \cite{LarsonSeparatingStationaryReflectionPrinciples} that ${<}\omega_2$-directed closed forcing preserves $\mathsf{MM}$. The proof can be adapted to show that ${<}\omega_2$-directed closed forcing preserves $\mathsf{MM}^{++}$ as well. We can now force $\LCCA$ by a ${<}\omega_2$-directed closed class forcing by starting the coding at $\omega_2$ so in the resulting generic extension $V[G]$, $\mathsf{MM}^{++}$ still holds. Further, all Woodin cardinals are preserved by this extension and by $\LCCA$, $\MPastempty{\Sigma_2}{\mathrm{SSP}}$ fails.
    \end{proof}

    This also solves Question 4 of Goodman's PhD Thesis \cite{GoodmanSigmaNCorrectForcingAxioms} in the negative, namely, $\mathsf{MM}^{++}$ is not equivalent to the forcing axiom $\Sigma_2$-$\mathsf{CMM}$ as defined in his thesis, because the latter implies $\MPast{\Sigma_2}{\mathrm{SSP}}{H_{\omega_2}}$.\footnote{In \cite{GoodmanSigmaNCorrectForcingAxioms}, the Maximality Principle for $\Sigma_n$ formulas was defined as $\MPast{\Sigma_n}{\mathcal{P}}{A}$.}

    Note that the proof of Theorem \ref{thm:MM_plus_Failure_of_Sigma_2-MP} works for virtually all standard forcing axioms, because they are also preserved by ${<}\omega_2$-directed closed forcings (cf.\ \cite{CoxForcingAxiomsAprrochability}). In particular, $\mathsf{PFA}+\neg\MPastempty{\Sigma_2}{\mathrm{proper}}$ is consistent with a proper class of large cardinals. On the other hand, the following question is still open.

    \begin{que}
        Is $\mathsf{PFA}\land\neg\MP{\Pi_2}{\mathrm{proper}}{H_{\omega_2}}$ consistent with a proper class of supercompact/extendible cardinals?
    \end{que}

    \bibliographystyle{alpha}
    \bibliography{MPbib}
    
\end{document}